\documentclass[a4paper]{amsart}
\usepackage[utf8]{inputenc}
\usepackage[T1]{fontenc}
\usepackage{lmodern}
\usepackage{amssymb}
\usepackage[all]{xy}
\usepackage{nicefrac,mathtools,enumitem}
\usepackage{microtype}

\usepackage[pdftitle={Homomorphisms of quantum groups},
  pdfauthor={Ralf Meyer, Sutanu Roy and Stanisław Lech Woronowicz},
  pdfsubject={Mathematics; MSC 19K35, 46L80}
]{hyperref}
\usepackage[lite]{amsrefs}
\newcommand*{\MRref}[2]{ \href{http://www.ams.org/mathscinet-getitem?mr=#1}{MR #1}}

\renewcommand{\PrintDOI}[1]{\href{http://dx.doi.org/#1}{DOI #1}%
  \IfEmptyBibField{volume}{, (to appear in print)}{}}

\numberwithin{equation}{section}

\theoremstyle{plain}
\newtheorem{theorem}[equation]{Theorem}
\newtheorem{lemma}[equation]{Lemma}
\newtheorem{proposition}[equation]{Proposition}

\newtheorem{corollary}[equation]{Corollary}

\theoremstyle{definition}
\newtheorem{definition}[equation]{Definition}

\theoremstyle{remark}
\newtheorem{remark}[equation]{Remark}

\newtheorem{example}[equation]{Example}

\newcommand*{\dom}{\mathcal D}% domain of an unbounded linear map

\newcommand*{\nb}{\nobreakdash}
\newcommand*{\Star}{*-}

\newcommand*{\C}{\mathbb C}

\newcommand*{\R}{\mathbb R}
\newcommand*{\N}{\mathbb N}

\newcommand*{\Bound}{\mathbb B}%adjointable operators on a Hilbert module
\newcommand*{\Comp}{\mathbb K}%compact operators on a Hilbert module

\newcommand*{\dd}{\textup d}%differential, used in dx in integrals
\newcommand*{\ima}{\textup i}%imaginary unit
\newcommand*{\univ}{\textup u}%universal
\newcommand*{\opp}{\textup{op}}%opposite
\newcommand*{\transpose}{\mathsf T}%transpose

\newcommand*{\Contvin}{\textup C_\textup 0}%continuous functions vanishing at infinity
\newcommand*{\Id}{\textup{id}}%identity map
\newcommand*{\Multunit}{\mathbb W}%muliplicative unitary on action on a hilbert space
\newcommand*{\multunit}{\textup W}%multiplicative unitary as a bicharacter of the multiplier algebra
\newcommand*{\maxcorep}{\mathcal V}%maximal corepresentation of C-hat
\newcommand*{\dumaxcorep}{\tilde{\mathcal V}}%maximal corepresentation of C
\newcommand*{\unibich}{\mathcal X}%universal bicharacter

\newcommand*{\Flip}{\Sigma}% flip operator on Hilbert space
\newcommand*{\flip}{\sigma}% flip map on the multiplier algebra
\newcommand*{\Cst}{\textup C^*}%C*-algebra
\newcommand*{\Cred}{\textup C^*_\textup r}%reduced group C*-algebra

\newcommand*{\Cstcat}{\mathfrak{C^*alg}}%category of C*-algebras
\newcommand*{\Forget}{\mathsf{For}}%forgetful functor

\newcommand*{\Hils}{\mathcal H}%Hilbert space
\newcommand*{\Mult}{\mathcal M}%multiplier algebra
\newcommand*{\U}{\mathcal U}%unitary group
\newcommand*{\Comult}{\Delta}%comultiplication
\newcommand*{\Linkunit}[2]{\textup V^{{#1}\rightarrow{#2}}}%Linking unitary

\newcommand*{\defeq}{\mathrel{\vcentcolon=}}

\newcommand*{\conj}[1]{\overline{#1}}

\begin{document}
\title{Homomorphisms of quantum groups}

\author{Ralf Meyer}
\email{rameyer@uni-math.gwdg.de}
\address{Mathematisches Institut and Courant Centre ``Higher order structures''\\
  Georg-August Universität Göttingen\\
  Bunsenstraße 3--5\\
  37073 Göttingen\\
  Germany}

\author{Sutanu Roy}
\email{sutanu@uni-math.gwdg.de}
\address{Mathematisches Institut\\
  Georg-August Universität Göttingen\\
  Bunsenstraße 3--5\\
  37073 Göttingen\\
  Germany}

\author{Stanisław Lech Woronowicz}
\email{Stanislaw.Woronowicz@fuw.edu.pl}
\address{Katedra Metod Matematycznych Fizyki\\
  Wydział Fizyki, Uniwersytet Warszawski\\
  Hoża 74\\
  00-682 Warszawa\\
  Poland}

\begin{abstract}
  In this article, we study several equivalent notions of homomorphism between locally compact quantum groups compatible with duality.  In particular, we show that our homomorphisms are equivalent to functors between the respective categories of coactions.  We lift the reduced bicharacter to universal quantum groups for any locally compact quantum group defined by a modular multiplicative unitary, without assuming Haar weights.  We work in the general setting of modular multiplicative unitaries.
\end{abstract}


\thanks{Supported by the German Research Foundation (Deutsche Forschungsgemeinschaft (DFG)) through the Research Training Group 1493 and the Institutional Strategy of the University of G\"ottingen, and by the Alexander von Humboldt-Stiftung.}

\maketitle

\section{Introduction}
\label{sec:introduction}

Let \((C,\Comult_C)\) and \((A,\Comult_A)\) be two \(\Cst\)\nb-bialgebras.  A \emph{Hopf \(^*\)\nb-homomorphism} between them is a morphism \(f\colon C\to A\) that intertwines the comultiplications, that is, the following diagram commutes:
\[
\xymatrix@C+1em{
  C \ar[d]_{\Comult_C}\ar[r]^f&A \ar[d]^{\Comult_A}\\
  C\otimes C \ar[r]_{f\otimes f}&A\otimes A.
}
\]
(Throughout this article, \(\Cst\)\nb-tensor products are spatial, and a \emph{morphism} between two \(\Cst\)\nb-algebras \(A\) and~\(B\) is a non-degenerate \Star{}homomorphism from~\(A\) to the multiplier algebra \(\Mult(B)\) or, equivalently, a strictly continuous unital \Star{}homomorphism \(\Mult(A)\to\Mult(B)\).  Thus \(\Cst\)\nb-algebras with the above morphisms form a category.)

These Hopf \(^*\)\nb-homomorphisms are the right morphisms between compact quantum groups (see~\cite{Wang:Free_products_of_CQG}) and, more generally, between amenable quantum groups.  For locally compact groups, however, they are not appropriate because they do not behave well for reduced group \(\Cst\)\nb-algebras.  It is easy to see that a group homomorphism \(f\colon G\to H\) induces Hopf \(^*\)\nb-homomorphisms from~\(\Contvin(H)\) to~\(\Contvin(G)\) and from~\(\Cst(G)\) to~\(\Cst(H)\).  But the latter does not always descend to the reduced group \(\Cst\)\nb-algebras.  For instance, the constant map from~\(G\) to the trivial group~\(\{1\}\) induces a Hopf \(^*\)\nb-homomorphism \(\Cred(G)\to \Cred(\{1\})=\C\) if and only if~\(G\) is amenable.

Thus Hopf \(^*\)\nb-homomorphisms are not compatible with duality, unless we use full duals everywhere: a Hopf \(^*\)\nb-homomorphism from~\(C\) to~\(A\) need not induce a Hopf \(^*\)\nb-homomorphism from~\(\hat{A}\) to~\(\hat{C}\).

More satisfactory notions of quantum group morphisms are introduced by Ng~\cite{Ng:Morph_of_Mult_unit} and later by Kustermans~\cite{Kustermans:LCQG_universal}.  While we wrote the first version of this article, we were unaware of Ng's article~\cite{Ng:Morph_of_Mult_unit} and, therefore, duplicated some of his work.

The theory in~\cite{Ng:Morph_of_Mult_unit} works well provided quantum groups are defined by multiplicative unitaries that lift to the universal quantum groups.  More precisely, the issue is to lift the (reduced) bicharacter in \(\U\Mult(\hat{C}\otimes C)\) to a bicharacter in \(\U\Mult(\hat{C}^\univ \otimes C^\univ)\), where \(\hat{C}^\univ\) and~\(C^\univ\) denote the universal quantum groups associated to the quantum groups \(\hat{C}\) and~\(C\), respectively.  (We write \(\U\Mult(A)\) for the group of unitary multipliers of a \(\Cst\)\nb-algebra~\(A\), and \(\U(\Hils)\) for the unitary group on a Hilbert space~\(\Hils\).)

Multiplicative unitaries that lift to \(\U\Mult(\hat{C}^\univ \otimes C^\univ)\) are called \emph{basic} in \cite{Ng:Morph_of_Mult_unit}*{Definition 2.3}, and several sufficient conditions for this are found in~\cite{Ng:Morph_of_Mult_unit}.  Kustermans~\cite{Kustermans:LCQG_universal} gets such a lifting from Haar weights.  Here we establish that all modular multiplicative unitaries are basic, so that the theory in~\cite{Ng:Morph_of_Mult_unit} works very generally.

The definition of locally compact quantum group we adopt here is the one by Sołtan and the third author based on modular multiplicative unitaries (see~\cite{Soltan-Woronowicz:Multiplicative_unitaries}) and not assuming Haar weights.  The description of quantum groups by multiplicative unitaries goes back to Baaj and Skandalis~\cite{Baaj-Skandalis:Unitaires}.

We now briefly list the equivalent notions of quantum group homomorphism that we study.  The most fundamental notion comes from the point of view that quantum groups encode symmetries of \(\Cst\)\nb-algebras, in the form of coactions.

Let  \(\Cstcat(A)\) be the category of \(\Cst\)\nb-algebras with a continuous coaction of \((A,\Comult_A)\), with \(A\)\nb-equivariant morphisms as arrows.  Forgetting the coaction provides a functor~\(\Forget\) to the category~\(\Cstcat\) of \(\Cst\)\nb-algebras without extra structure.  A \emph{quantum group homomorphism} from~\((C,\Comult_C)\) to~\((A,\Comult_A)\) is a functor
\begin{equation}
  \label{eq:intro_functor_morphism}
  F\colon \Cstcat(C)\to\Cstcat(A)\qquad \text{with \(\Forget\circ F=\Forget\).}  
\end{equation}
A Hopf \(^*\)\nb-homomorphism \(f\colon C\to A\) clearly induces such a functor: compose a coaction \(D\to D\otimes C\) with \(\Id_D\otimes f\) to get a coaction of~\(A\).  While the definition in~\eqref{eq:intro_functor_morphism} is conceptually clear, it is hard to work with.  We provide several more concrete descriptions of quantum group homomorphisms.

A Hopf \(^*\)\nb-homomorphism \(f\colon C\to A\) yields a unitary multiplier
\[
V = V_f \defeq (\Id_{\hat{C}} \otimes f)(\multunit^C)
\in \U\Mult(\hat{C}\otimes A),
\]
which is a \emph{bicharacter}, that is,
\begin{equation}
  \label{eq:intro_bicharacter}
  (\Comult_{\hat{C}}\otimes\Id_A)V=V_{23}V_{13}
  \quad\text{and}\quad
  (\Id_{\hat{C}}\otimes\Comult_A)V=V_{12}V_{13}.
\end{equation}
We show that bicharacters in \(\U\Mult(\hat{C}\otimes A)\) correspond to quantum group homomorphisms from~\(C\) to~\(A\) and therefore call them bicharacters \emph{from~\(C\) to~\(A\)}.  Bicharacters were already interpreted as homomorphisms of quantum groups in~\cite{Ng:Morph_of_Mult_unit}, where they are called \emph{\(\Multunit_C\)\nb-\(\Multunit_A\)-birepresentations}.

Bicharacters are nicely compatible with duality.  If~\(V\) is a bicharacter from~\(C\) to~\(A\), then \(\sigma(V^*)\) is a bicharacter from~\(\hat{A}\) to~\(\hat{C}\), where we use the coordinate flip
\[
\sigma\colon \hat{C}\otimes A\to \hat{\hat{A}}\otimes\hat{C}.
\]

A Hopf \(^*\)\nb-homomorphism \(f\colon C\to A\) is determined by the bicharacter~\(V_f\) because elements of the form \((\omega\otimes\Id_C)(\multunit^C)\) for linear functionals~\(\omega\) on~\(\hat{C}\) are dense in~\(C\) and \(f\bigl((\omega\otimes\Id_C)(\multunit^C)\bigr) = (\omega\otimes\Id_A)(V_f)\).  Furthermore, if~\(f\) admits a dual quantum group homomorphism \(\hat{f}\colon \hat{A}\to\hat{C}\), then these two are related by
\begin{equation}
  \label{eq:dual_of_strong}
  (\hat{f}\otimes\Id_A)(\multunit^A)
  = V_f
  = (\Id_{\hat{C}} \otimes f)(\multunit^C).
\end{equation}

The standard Hilbert space representations of \(\hat{C}\) and~\(A\) turn a bicharacter~\(V\) into a unitary operator \(\mathbb{V}=\mathbb{V}_f\) on \(\Hils_C\otimes\Hils_A\) that satisfies the two pentagonal equations
\begin{equation}
  \label{eq:intro_pentagonal}
  \mathbb{V}_{23}\Multunit^C_{12} = \Multunit^C_{12}\mathbb{V}_{13}\mathbb{V}_{23}
  \quad\text{and}\quad
  \Multunit_{23}^A \mathbb{V}_{12} = \mathbb{V}_{12}\mathbb{V}_{13}\Multunit_{23}^A,
\end{equation}
which relate it to the multiplicative unitaries \(\Multunit^C\) and~\(\Multunit^A\) of \(C\) and~\(A\).  Conversely, such unitary operators come from bicharacters, so that we call them \emph{concrete bicharacters}.

We show that concrete bicharacters and hence bicharacters form a category.  The composition of two bicharacters \(\Linkunit{C}{A}\in\U\Mult(\hat{C}\otimes A)\) and \(\Linkunit{A}{B}\in\U\Mult(\hat{A}\otimes B)\) is the unique \(\Linkunit{C}{B}\in\U\Mult(\hat{C}\otimes B)\) that satisfies
\[
\mathbb{V}^{A\to B}_{23}\mathbb{V}^{C\to A}_{12}
= \mathbb{V}^{C\to A}_{12}\mathbb{V}^{C\to B}_{13}\mathbb{V}^{A\to B}_{23}.
\]
As expected, the composition of a bicharacter \(\Linkunit{C}{A}\in\U\Mult(\hat{C}\otimes A)\) with the bicharacter~\(V_f\) of a Hopf \(^*\)\nb-homomorphism \(f\colon A\to B\) yields \((\Id_{\hat{C}}\otimes f) (\Linkunit{C}{A}\)).

Previously, it was suggested to define quantum group homomorphisms by passing to universal quantum groups~\cite{Vaes:Induction_Imprimitivity}. Moreover, Kustermans~\cite{Kustermans:LCQG_universal} shows that the Hopf \(^*\)\nb-homomorphisms between universal quantum groups correspond to certain coactions on the von Neumann algebraic versions of these quantum groups.

The universal property that defines the universal quantum group \((C^\univ,\Comult^\univ)\) of a quantum group \((C,\Comult)\) immediately implies that bicharacters from~\(C\) to~\(A\) correspond bijectively to Hopf \(^*\)\nb-homomorphisms \(C^\univ\to A\).  We show that such bicharacters lift uniquely to bicharacters from~\(C\) to~\(A^\univ\).  Thus quantum group homomorphisms \(C\to A\) are in bijection with Hopf \(^*\)\nb-homomorphisms \(C^\univ\to A^\univ\).  The composition of bicharacters corresponds to the usual composition of Hopf \(^*\)\nb-homomorphisms between universal quantum groups.

There is a bijective correspondence between bicharacters and certain coactions.  A \emph{right quantum group homomorphism} from~\(C\) to~\(A\) is a morphism \(\Delta_R\colon C\to C\otimes A\) such that the following two diagrams commute:
\begin{equation}
  \label{eq:intro_right_homomorphism}
  \begin{gathered}
    \xymatrix@C+1.2em{
      C \ar[r]^{\Delta_R}\ar[d]_{\Comult_C}&
      C\otimes A \ar[d]^{\Comult_C\otimes\Id_A}\\
      C\otimes C \ar[r]_-{\Id_C\otimes \Delta_R}&
      C\otimes C\otimes A,
    }\qquad
    \xymatrix@C+1.2em{
      C \ar[r]^{\Delta_R}\ar[d]_{\Delta_R}&
      C\otimes A \ar[d]^{\Id_C\otimes \Comult_A}\\
      C\otimes A \ar[r]_-{\Delta_R\otimes\Id_A}&
      C\otimes A\otimes A.
    }
  \end{gathered}
\end{equation}
A bicharacter \(V\in\U\Mult(\hat{C}\otimes A)\) yields a right quantum group homomorphism
\[
\Delta_R\colon C\to C\otimes A,\qquad
\Delta_R(x) \defeq \mathbb{V}(x\otimes1)\mathbb{V}^*\quad\text{for all \(x\in C\).}
\]
The bicharacter corresponding to~\(\Delta_R\) is the unique unitary \(V\in\U\Mult(\hat{C}\otimes A)\) with
\begin{equation}
  \label{eq:intro_bicharacter_from_right_homomorphism}
  (\Id_{\hat{C}} \otimes \Delta_R)(\multunit) = \multunit_{12}V_{13}.  
\end{equation}
Given a functor \(F\colon \Cstcat(C)\to\Cstcat(A)\) with \(\Forget\circ F=\Forget\), we get such a right coaction~\(\Delta_R\) by applying~\(F\) to the coaction~\(\Comult_C\) on~\(C\).  In this way, right quantum group homomorphisms are equivalent to functors as in~\eqref{eq:intro_functor_morphism}.

The following technical result is a crucial tool in this article.  If \(a,b\in \Bound (\Hils_A)\) satisfy \(\Multunit (a\otimes 1) = (1\otimes b)\Multunit\), then already \(a=b\in\C\cdot 1\).  This implies that a multiplier of~\(A\) is constant if it is left or right invariant.  This result is already known in the presence of Haar weights.  We establish it in the more general framework of~\cite{Soltan-Woronowicz:Multiplicative_unitaries}.

\section{Invariants are constant}
\label{sec:invariants_constant}

Let \((C,\Comult_C)\) and \((A,\Comult_A)\) be two quantum groups in the sense of~\cite{Soltan-Woronowicz:Multiplicative_unitaries}.  That is, they are obtained from modular multiplicative unitaries \(\Multunit^C\in\U(\Hils_C\otimes\Hils_C)\) and \(\Multunit^A\in\U(\Hils_A\otimes\Hils_A)\) for certain Hilbert spaces \(\Hils_C\) and~\(\Hils_A\).  Let \(\multunit^C\in\U\Mult(\hat{C}\otimes C)\) and \(\multunit^A\in\U\Mult(\hat{A}\otimes A)\) be their reduced bicharacters.

Being constructed from modular multiplicative unitaries, our locally compact quantum groups are concrete \(\Cst\)\nb-algebras, represented on some Hilbert space.  However, several non-equivalent multiplicative unitaries may give the same locally compact quantum group, that is, isomorphic pairs \((C,\Comult_C)\).  Therefore, we distinguish between elements of \(\Cst\)\nb-algebras such as \(\hat{C}\otimes C\) and the Hilbert space operators they generate in the representation of \(\hat{C}\otimes C\).  For a unitary multiplier~\(U\) of \(\hat{C}\otimes C\), we write~\(\mathbb{U}\) for~\(U\) considered as an operator on the Hilbert space~\(\Hils_C\otimes\Hils_C\), where \(\hat{C},C\subseteq\Bound(\Hils_C)\).  In particular, this applies to the reduced bicharacter \(\multunit^C\in\U\Mult(\hat{C}\otimes C)\) and the modular multiplicative unitary \(\Multunit^C\in\U(\Hils_C\otimes\Hils_C)\).  Whereas~\(\multunit^C\) is uniquely determined by \((C,\Comult_C)\), \(\Multunit^C\) is not unique in any sense.

Recall the following properties:
\begin{alignat}{2}
  \label{eq:pentagon}
  \Multunit^C_{23}\Multunit^C_{12}
  &= \Multunit^C_{12}\Multunit^C_{13}\Multunit^C_{23}
  &\qquad&
  \text{in \(\U(\Hils_C\otimes \Hils_C \otimes \Hils_C)\),}\\
  \label{eq:Delta_via_W}
  \Comult_C(x)
  &= \Multunit^C(x\otimes 1)(\Multunit^C)^*
  &\qquad&
  \text{in \(\Bound(\Hils_C\otimes \Hils_C)\) for all \(x\in C\),}\\
  \label{eq:hat_Delta_via_W}
  \Comult_{\hat{C}}(y)
  &= \Flip(\Multunit^C)^*(1\otimes y)\Multunit^C\Flip
  &\qquad&
  \text{in \(\Bound(\Hils_C\otimes \Hils_C)\) for all \(y\in \hat{C}\),}\\
  \label{eq:Delta_W}
  (\Id_{\hat{C}}\otimes \Comult_C)(\multunit^C)
  &= \multunit^C_{12}\multunit^C_{13}
  &\qquad& \text{in \(\hat{C}\otimes C\otimes C\),}\\
  \label{eq:hat_Delta_W}
  (\Comult_{\hat{C}}\otimes \Id_C)(\multunit^C)
  &= \multunit^C_{23}\multunit^C_{13}
  &\qquad& \text{in \(\hat{C}\otimes \hat{C}\otimes C\).}
\end{alignat}
Here \(\Sigma\in\U(\Hils_C\otimes\Hils_C)\) denotes the coordinate flip.

\begin{theorem}
 \label{the:co-invariant_type_operators}
 Let~\(\Hils\) be a Hilbert space and let \(\Multunit\in \Bound(\Hils\otimes\Hils)\) be a modular multiplicative unitary.  If \(a,b\in \Bound(\Hils)\) satisfy \(\Multunit(a\otimes 1)=(1\otimes b)\Multunit\), then \(a=b=\lambda 1\) for some \(\lambda\in\C\).  More generally, if \(a,b\in \Mult(\Comp(\Hils)\otimes D)\) for some \(\Cst\)\nb-algebra~\(D\) satisfy \(\Multunit_{12} a_{13} = b_{23}\Multunit_{12}\), then \(a=b\in \C\cdot 1_\Hils\otimes \Mult(D)\).
\end{theorem}

\begin{proof}
  Define the operators \(\hat{Q}\), \(Q\), and~\(\widetilde{\Multunit}\) as in \cite{Soltan-Woronowicz:Remark_manageable}*{Definition 2.1}.  First we prove the assertion without~\(D\) and under the additional assumption \(b^*\dom(Q)\subseteq \dom(Q)\).  Our assumption \(\Multunit(a\otimes 1)=(1\otimes b)\Multunit\) means
 \[
 (x\otimes y \mid \Multunit \mid az\otimes u)
 = (x\otimes b^*y\mid \Multunit \mid z\otimes u)
 \]
 for all \(x, z\in\Hils\), \(y\in\dom(Q)\) and \(u\in\dom(Q^{-1})\).  The modularity condition for~\(\Multunit\) yields
 \[
   \bigl(\conj{a z}\otimes Qy \bigm| \widetilde{\Multunit} \bigm| \conj{x}\otimes Q^{-1}u \bigr)
    = \bigl(\conj{z}\otimes Qb^* y \bigm| \widetilde{\Multunit} \bigm| \conj{x}\otimes Q^{-1}u\bigr).
  \]
  In this formula, \(\widetilde{\Multunit}(\conj{x}\otimes Q^{-1}u)\) runs through a dense subset of \(\conj{\Hils}\otimes\Hils\).  Therefore, we may replace \(\widetilde{\Multunit}(\conj{x}\otimes Q^{-1}u)\) by \(\conj{x}\otimes Q^{-1}u\) and get
  \[
  (\conj{a z}\otimes Qy \mid \conj{x}\otimes Q^{-1}u)
  = (\conj{z}\otimes Qb^* y \mid \conj{x}\otimes Q^{-1}u),
  \]
  Now we reverse the above computation without \(\Multunit\) and~\(\widetilde{\Multunit}\), using the self-adjointness of~\(Q\).  This yields
  \[
  (x\otimes y \mid az\otimes u)
  = (x\otimes b^*y\mid z\otimes u)
  \]
  for all \(x,y,z,u\in\Hils\).  Hence \(a\otimes 1=1\otimes b\), so that \(a=b\in\C\cdot1\).

  To remove the assumption \(b^*\dom(Q)\subseteq \dom(Q)\), we regularise \(a\) and~\(b\).  For \(a\in\Bound(\Hils)\) and \(n\in\N\), we define
  \[
  \widehat{R}_n(a) \defeq \int_{-\infty}^{+\infty} \widehat{Q}^{-\ima t}a\widehat{Q}^{\ima t}\delta_n(t) \,\dd t
  \quad\text{and}\quad
  R_n(a) \defeq \int_{-\infty}^{+\infty} Q^{-\ima t} a Q^{\ima t} \delta_n(t) \,\dd t,
  \]
  where
  \[
  \delta_n(t) \defeq \sqrt{\frac{n}{2\pi}} \exp\left(-\frac{nt^2}{2}\right)
  \]
  is a \(\delta\)\nb-like sequence of Gaussian functions.  Since
  \[
  \Multunit^*(\widehat{Q}\otimes Q)\Multunit=\widehat{Q}\otimes Q,
  \]
  our condition \(\Multunit(a\otimes 1)=(1\otimes b)\Multunit\) implies
  \[
  \Multunit(\widehat{R}_n(a)\otimes 1) =(1\otimes R_n(b))\Multunit.
  \]
  We will show below that
  \begin{equation}
    \label{eq:domain_condition}
    R_n(b)^*\dom(Q)\subseteq \dom(Q).
  \end{equation}
  The first part of the proof now yields \(\widehat{R}_n(a)=R_n(b)=\lambda_n1\) for all \(n\in\N\).  If \(n\to\infty\), then \(\widehat{R}_n(a)\) and \(R_n(b)\) converge weakly towards \(a\) and~\(b\), respectively.  Hence we get \(a=b=\lambda 1\) for some \(\lambda\in\C\) in full generality.

  It remains to establish~\eqref{eq:domain_condition}.  Let \(x,y\in\dom(Q)\).  Then the function
  \[
  f_{x,y}(z) \defeq (Q^{\ima(\conj{z}-\ima)} x\mid b^*\mid Q^{\ima z}y)
  \]
  is well-defined, bounded, and continuous in the strip \(\Sigma\defeq \{z\in\C\colon -1\le \operatorname{Im} z\le 0\}\) and holomorphic in the interior of~\(\Sigma\).  In particular, for \(t\in\R\colon\)
  \begin{equation}
    \label{eq:holomorphic_extension}
    f_{x,y}(t) = (Qx\mid Q^{-\ima t}b^*Q^{\ima t} \mid y),\qquad
    f_{x,y}(t-\ima) = (x\mid Q^{-\ima t}b^*Q^{\ima t} \mid Qy).
  \end{equation}
  By Cauchy's Theorem, the integrals of \(f_{x,y}(z)\delta_n(z)\) along the lines \(\R+\ima s\) for \(0\le 1\le s\) do not depend on~\(s\).  For \(s=0\) and \(s=1\), \eqref{eq:holomorphic_extension} shows that the integrals are \((Qx\mid R_n(b)^*\mid y)\) and
  \[
  \biggl( x \biggm| \int_{-\infty}^{+\infty} Q^{-\ima t}b^*Q^{\ima t}\delta_n(t-\ima)\,\dd t \biggm| Qy \biggr),
  \]
  respectively.  Their equality shows that \((Qx\mid R_n(b)^* y)\) depends continuously on~\(x\).  This yields \(R_n(b)^*y\in\dom(Q^*)=\dom(Q)\), that is, \eqref{eq:domain_condition}.

  Finally, we add the coefficient algebra~\(D\).  If \(a,b\in\Mult(\Comp(\Hils)\otimes D)\) satisfy \(\Multunit_{12} a_{13}= b_{23}\Multunit_{12}\) in \(\Mult(\Comp(\Hils\otimes\Hils)\otimes D)\), then the first part of the theorem applies to the slices \((\Id\otimes\mu)(a)\) and \((\Id\otimes\mu)(b)\) for all \(\mu\in D'\).  Thus \((\Id\otimes\mu)(a) = (\Id\otimes\mu)(b) = \lambda_\mu\cdot 1\) for all \(\mu\in D'\).  This implies that \(a=b\in \C\cdot 1 \otimes\Mult(D)\).
\end{proof}

\begin{corollary}
  \label{cor:C_co-invariant_are_scalar_multiple}
  Let \((C,\Comult_C)\) be a quantum group constructed from a manageable \textup(or, more generally, from a modular\textup) multiplicative unitary \(\Multunit\in\Bound(\Hils\otimes\Hils)\).  If \(c\in\Mult(C)\), then \(\Comult_C(c)\in\Mult(C\otimes 1)\) or \(\Comult_C(c)\in\Mult(1\otimes C)\) if and only if \(c\in\C\cdot1\).

  More generally, if~\(D\) is a \(\Cst\)\nb-algebra and \(c\in \Mult(C\otimes D)\), then \((\Comult_C\otimes\Id_D)(c) \in\Mult(C\otimes 1\otimes D)\) or \((\Comult_C\otimes\Id_D)(c) \in\Mult(1\otimes C\otimes D)\) if and only if \(c\in\C\cdot1 \otimes \Mult(D)\).
\end{corollary}

\begin{proof}
  Using~\eqref{eq:Delta_via_W}, we rewrite the equation \(\Comult_C(c) = 1\otimes c'\) for \(c,c'\in\Mult(C\otimes D)\) as \(\Multunit^C_{12}c_{13}=c'_{23}\Multunit^C_{12}\).  Now Theorem~\ref{the:co-invariant_type_operators} yields \(c\in\C\cdot1 \otimes \Mult(D)\).  If \(\Comult_C(c)=c'\otimes 1\) instead, then we apply the unitary antipodes.  With \(a \defeq (R_C\otimes\Id_D)(c)\) and \(a' \defeq (R_C\otimes\Id_D)(c')\), we get \(\Comult_C(a) = 1\otimes a'\).  The argument above shows \(a\in\C\cdot 1\otimes \Mult(D)\) and hence \(c\in\C\cdot 1\otimes\Mult(D)\).
\end{proof}

\section{Bicharacters}
\label{sec:linking_unitaries}

\begin{definition}
  \label{def:linking_unitary}
  A unitary \(V\in\U\Mult(\hat{C}\otimes A)\) is called a \emph{bicharacter from \(C\) to~\(A\)} if
  \begin{alignat}{2}
    \label{eq:linking_unitary_comult_C_hat}
    (\Comult_{\hat{C}}\otimes\Id_A)V
    &=V_{23}V_{13}
    &\qquad &\text{in \(\U\Mult(\hat{C}\otimes\hat{C}\otimes A)\)},\\
    \label{eq:linking_unitary_comult_A}
    (\Id_{\hat{C}}\otimes\Comult_A)V
    &=V_{12}V_{13}
    &\qquad &\text{in \(\U\Mult(\hat{C}\otimes A\otimes A)\)}.
  \end{alignat}
\end{definition}

\begin{lemma}
  \label{lemm:equiv_criterion_of_linkunit}
   A unitary \(\mathbb{V}\in \U(\Hils_C\otimes \Hils_A)\) comes from a bicharacter \(V\in\U\Mult(\hat{C}\otimes A)\) \textup(which is necesssarily unique\textup) if and only if
  \begin{alignat}{2}
    \label{eq:linking_unitary_pentagon_C}
    \mathbb{V}_{23}\Multunit^C_{12}
    &= \Multunit^C_{12}\mathbb{V}_{13}\mathbb{V}_{23}
    &\qquad &
    \text{in \(\U(\Hils_C\otimes\Hils_C\otimes\Hils_A)\)},\\
    \label{eq:linking_unitary_pentagon_A}
    \Multunit_{23}^A \mathbb{V}_{12}
    &= \mathbb{V}_{12}\mathbb{V}_{13}\Multunit_{23}^A
    &\qquad &
    \text{in \(\U(\Hils_C\otimes\Hils_A\otimes\Hils_A)\)}.
  \end{alignat}
\end{lemma}

\begin{proof}
  The representation of \(\hat{C}\otimes\hat{C}\otimes A\) on \(\Hils_C\otimes\Hils_C\otimes\Hils_A\) is faithful.  Hence a bicharacter \(V\in\U\Mult(\hat{C}\otimes A)\) is determined by its image \(\mathbb{V}\in\U(\Hils_C\otimes\Hils_A)\), and \eqref{eq:linking_unitary_comult_C_hat} and~\eqref{eq:linking_unitary_comult_A} are equivalent equations of unitary operators on \(\Hils_C\otimes\Hils_C\otimes\Hils_A\).  Using~\eqref{eq:hat_Delta_via_W}, we rewrite~\eqref{eq:linking_unitary_comult_C_hat} as \(\Flip_{12}(\Multunit^C_{12})^* \mathbb{V}_{23} \Multunit^C_{12}\Flip_{12} = \mathbb{V}_{23}\mathbb{V}_{13}\).  This is equivalent to~\eqref{eq:linking_unitary_pentagon_C}.  A similar argument shows that~\eqref{eq:linking_unitary_comult_A} is equivalent to~\eqref{eq:linking_unitary_pentagon_A}.

  It remains to show that a unitary~\(\mathbb{V}\) on \(\Hils_C\otimes\Hils_A\) that satisfies \eqref{eq:linking_unitary_pentagon_C} and~\eqref{eq:linking_unitary_pentagon_A} necessarily belongs to \(\Mult(\hat{C}\otimes A)\).  We argue as in the proof of \cite{Woronowicz:Multiplicative_Unitaries_to_Quantum_grp}*{Theorem 1.6.2}.  The unitary~\(\mathbb{V}\) is adapted to~\(\Multunit^A\) in the sense of~\cite{Woronowicz:Multiplicative_Unitaries_to_Quantum_grp} by~\eqref{eq:linking_unitary_pentagon_A}.  Rewriting~\eqref{eq:linking_unitary_pentagon_A} as \(\mathbb{V}_{13} = \mathbb{V}_{12}^* \Multunit_{23}^A \mathbb{V}_{12}(\Multunit_{23}^A)^*\), we see that \(\mathbb{V}\in\Mult(\Comp(\Hils_C)\otimes A)\).  Then~\eqref{eq:linking_unitary_pentagon_C} in the form \(\mathbb{V}_{13} = (\Multunit^C_{12})^*\mathbb{V}_{23}\Multunit^C_{12}\mathbb{V}_{23}^*\) shows that \(\mathbb{V}_{13}\in\Mult(\hat{C}\otimes\Comp(\Hils_C)\otimes A)\), so that \(\mathbb{V}\in\Mult(\hat{C}\otimes A)\) as asserted.
\end{proof}

\begin{example}
  \label{ex:strong_lu}
  Equations \eqref{eq:Delta_W}--\eqref{eq:hat_Delta_W} show that a Hopf \(^*\)\nb-homomorphism \(f\colon C\to A\) yields a bicharacter \(V_f\defeq (\Id_{\hat{C}}\otimes f)\multunit^C\).  In particular, \(\multunit^C\) is a bicharacter.
\end{example}

\begin{remark}
  \label{rem:linking_unitary_indep_of_mult_unit}
  The criterion in Lemma~\ref{lemm:equiv_criterion_of_linkunit} has the merit of using only the language of multiplicative unitaries and pentagon equations.  But the same quantum group may be generated by different multiplicative unitaries.  Since~\(\multunit^C\) only depends on \((C,\Comult_C)\) by~\cite{Soltan-Woronowicz:Multiplicative_unitaries}, bicharacters from~\(C\) to~\(A\) depend only on \((A,\Comult_A)\) and \((C,\Comult_C)\).
\end{remark}

Now we define the composition of (concrete) bicharacters as in \cite{Ng:Morph_of_Mult_unit}*{Lemma 2.5}.  Let \((B,\Comult_B)\) be another quantum group.

\begin{definition}
  \label{def:composition_of_linking_unitary}
  A unitary \(\Linkunit{C}{B}\in\U\Mult(\hat{C}\otimes B)\) is called a \emph{composition} of two bicharacters \(\Linkunit{C}{A}\in\U\Mult(\hat{C}\otimes A)\) and \(\Linkunit{A}{B}\in\U\Mult(\hat{A}\otimes B)\) if its image \(\mathbb{V}^{C\to B}\) in \(\U(\Hils_C\otimes \Hils_B)\) satisfies
  \[
  \mathbb{V}^{A\to B}_{23}\mathbb{V}^{C\to A}_{12}
  =\mathbb{V}^{C\to A}_{12}\mathbb{V}^{C\to B}_{13}\mathbb{V}^{A\to B}_{23}
  \qquad \text{in \(\U(\Hils_C\otimes\Hils_A\otimes\Hils_B)\)}.
  \]
  or, equivalently,
  \begin{equation}
    \label{eq:link_unit_comp_exist}
    \mathbb{V}^{C\to B}_{13}
    = (\mathbb{V}^{C\to A}_{12})^* \mathbb{V}^{A\to B}_{23}
    \mathbb{V}^{C\to A}_{12} (\mathbb{V}^{A\to B}_{23})^*.
  \end{equation}
\end{definition}

We also briefly write \(\Linkunit{C}{B}= \Linkunit{A}{B} * \Linkunit{C}{A}\).

\begin{lemma}
  \label{lem:composition_exists}
  For any two bicharacters \(\Linkunit{C}{A}\) and \(\Linkunit{A}{B}\), there is a unique composition~\(\Linkunit{C}{B}\), and it is a bicharacter from~\(C\) to~\(B\).
\end{lemma}

\begin{proof}
  Define \(\tilde{V}\defeq (\mathbb{V}^{C\to A}_{12})^* \mathbb{V}^{A\to B}_{23} \mathbb{V}^{C\to A}_{12} (\mathbb{V}^{A\to B}_{23})^* \in \U\Mult(\hat{C}\otimes \Comp(\Hils_B)\otimes B)\).  We are going to use Theorem~\ref{the:co-invariant_type_operators} to show that \(\tilde{V}\in\U\Mult(\hat{C}\otimes 1\otimes B)\).
  \begin{align*}
    & \Multunit^A_{23} \tilde{V}_{124} (\Multunit^A_{23})^*\\
    & =  \Multunit^A_{23}(\mathbb{V}^{C\to A}_{12})^*
    \mathbb{V}^{A\to B}_{24}\mathbb{V}^{C\to A}_{12}
    (\mathbb{V}^{A\to B}_{24})^*(\Multunit^A_{23})^*\\
    & = (\mathbb{V}^{C\to A}_{13})^*(\mathbb{V}^{C\to A}_{12})^*
    \Multunit_{23}^A\mathbb{V}^{A\to B}_{24}
    \mathbb{V}^{C\to A}_{12}(\mathbb{V}^{A\to B}_{24})^*(\Multunit^A_{23})^*\\
    & = (\mathbb{V}^{C\to A}_{13})^*(\mathbb{V}^{C\to A}_{12})^*
    \mathbb{V}^{A\to B}_{34}\Multunit^A_{23}(\mathbb{V}^{A\to B}_{34})^*
    \mathbb{V}_{12}^{C\to A}\mathbb{V}^{A\to B}_{34}(\Multunit^A_{23})^*
    (\mathbb{V}^{A\to B}_{34})^*\\
    & = (\mathbb{V}^{C\to A}_{13})^*(\mathbb{V}^{C\to A}_{12})^*
    \mathbb{V}^{A\to B}_{34}\Multunit^A_{23}\mathbb{V}^{C\to A}_{12}
    (\Multunit^A_{23})^*(\mathbb{V}^{A\to B}_{34})^*\\
    & = (\mathbb{V}^{C\to A}_{13})^*(\mathbb{V}^{C\to A}_{12})^*
    \mathbb{V}^{A\to B}_{34}\mathbb{V}^{C\to A}_{12}\mathbb{V}^{C\to A}_{13}
    (\mathbb{V}^{A\to B}_{34})^*\\
    & = (\mathbb{V}^{C\to A}_{13})^*\mathbb{V}^{A\to B}_{34}\mathbb{V}^{C\to A}_{13}
    (\mathbb{V}^{A\to B}_{34})^* = \tilde{V}_{134};
  \end{align*}
  the first step uses~\eqref{eq:linking_unitary_pentagon_A}; the second step uses~\eqref{eq:linking_unitary_pentagon_C} for \(\Linkunit{A}{B}\); the third step uses that \(\mathbb{V}^{A\to B}_{34}\) and~\(\mathbb{V}^{C\to A}_{12}\) commute; the fourth step again uses~\eqref{eq:linking_unitary_pentagon_A}; and the last step uses again that \(\mathbb{V}^{A\to B}_{34}\) and~\(\mathbb{V}^{C\to A}_{12}\) commute.  Now Theorem~\ref{the:co-invariant_type_operators} shows that \(V\in\U\Mult(\hat{C}\otimes 1\otimes B)\).  This is the unique solution for~\eqref{eq:link_unit_comp_exist}.

  The following computation yields~\eqref{eq:linking_unitary_pentagon_C} for~\(\tilde{V}\):
  \begin{align*}
    & \Flip_{12}(\Multunit^C_{12})^* \mathbb{V}^{C\to B}_{24} \Multunit^C_{12}\Flip_{12}\\
    & = \Flip_{12}(\Multunit^C_{12})^*(\mathbb{V}^{C\to A}_{23})^*
    \mathbb{V}^{A\to B}_{34}\mathbb{V}^{C\to A}_{23}
    (\mathbb{V}^{A\to B}_{34})^*\Multunit^C_{12}\Flip_{12}\\
    & = \Flip_{12}(\mathbb{V}^{C\to A}_{23})^*
    (\mathbb{V}^{C\to A}_{13})^*(\Multunit^C_{12})^*
    \mathbb{V}^{A\to B}_{34}\mathbb{V}^{C\to A}_{23}
    (\mathbb{V}^{A\to B}_{34})^*\Multunit^C_{12}\Flip_{12}\\
    & = (\mathbb{V}^{C\to A}_{13})^*(\mathbb{V}^{C\to A}_{23})^*
    \mathbb{V}^{A\to B}_{34}\Flip_{12}(\Multunit^C_{12})^*
    \mathbb{V}^{C\to A}_{23}\Multunit^C_{12}\Flip_{12}
    (\mathbb{V}^{A\to B}_{34})^*\\
    & = (\mathbb{V}^{C\to A}_{13})^*(\mathbb{V}^{C\to A}_{23})^*
    \mathbb{V}^{A\to B}_{34}\mathbb{V}^{C\to A}_{23}
    \mathbb{V}^{C\to A}_{13}(\mathbb{V}^{A\to B}_{34})^*\\
    & = (\mathbb{V}^{C\to A}_{13})^*\mathbb{V}^{C\to B}_{24}
    \mathbb{V}^{A\to B}_{34}\mathbb{V}^{C\to A}_{13}
    (\mathbb{V}^{A\to B}_{34})^*\\
    & = \mathbb{V}^{C\to B}_{24}(\mathbb{V}^{C\to A}_{13})^*
    \mathbb{V}^{A\to B}_{34}\mathbb{V}^{C\to A}_{13}
    (\mathbb{V}^{A\to B}_{34})^*
    = \mathbb{V}^{C\to B}_{24}\mathbb{V}^{C\to B}_{14}.
  \end{align*}
  The first step uses~\eqref{eq:linking_unitary_pentagon_C}; the second step uses properties of~\(\Flip\) and that \(\Multunit^C_{12}\) and~\(\mathbb{V}^{A\to B}_{34}\) commute; the third step uses~\eqref{eq:linking_unitary_pentagon_C}; the fourth step uses~\eqref{eq:link_unit_comp_exist}; the fifth step uses that \(\mathbb{V}^{C\to A}_{13}\) and~\(\mathbb{V}^{C\to A}_{24}\) commute; and the last step uses~\eqref{eq:link_unit_comp_exist}.

  Similarly, one shows~\eqref{eq:linking_unitary_pentagon_A}.  Hence~\(\Linkunit{C}{B}\) is indeed a bicharacter.
\end{proof}

\begin{remark}
  \label{rem:composition_welldefined}
  The composition of bicharacters at first sight depends on the choice of the generating modular multiplicative unitaries.  Both Theorem~\ref{the:linking_universal} and Proposition~\ref{pro:compose_right} below provide alternative descriptions of this composition that show its independence of auxiliary choices.
\end{remark}

\begin{proposition}
  \label{pro:associative}
  The composition of bicharacters is associative, and the multiplicative unitary~\(\multunit^C\) is an identity on~\(C\).  Thus bicharacters with the above composition and locally compact quantum groups are the arrows and objects of a category, called bicharacter category.
\end{proposition}

\begin{proof}
  Only associativity of the composition is non-trivial.  This may be establised by a direct composition similar to the ones above.  We omit it because associativity follows immediately from Theorem~\ref{the:linking_universal} or from Proposition~\ref{pro:compose_right} below, which translate the composition into a diffrent language where associativity is obvious.
\end{proof}

Recall that the dual of a multiplicative unitary~\(\Multunit\) is the multiplicative unitary \(\hat{\Multunit}\defeq\Flip\Multunit^*\Flip\).  Correspondingly, the reduced bicharacter of the dual quantum group is \(\hat{\multunit}\defeq\flip(\multunit^*)\).  Here \(\flip\colon \hat{C}\otimes C\to C\otimes\hat{C}\) is the tensor flip automorphism and \(\Flip\colon \Hils_C\otimes\Hils_C\to \Hils_C\otimes\Hils_C\) is the tensor flip unitary.  A similar duality works for all bicharacters:

\begin{proposition}
  \label{pro:dual_bicharacter}
  Let \(V\in\U\Mult(\hat{C}\otimes A)\) be a bicharacter from~\(C\) to~\(A\) and let \(\mathbb{V}\in\U(\Hils_C\otimes\Hils_A)\) be the corresponding concrete bicharacter.  Then
  \[
  \hat{V}\defeq\flip(V^*) \in \U\Mult(A\otimes\hat{C})\quad\text{and}\quad
  \hat{\mathbb{V}}\defeq\Flip\mathbb{V^*}\Flip\in\U(\Hils_A\otimes\Hils_C)
  \]
  are a bicharacter from~\(\hat{A}\) to~\(\hat{C}\) and the correponding concrete bicharacter.  Here we identify the double dual of \((A,\Comult_A)\) again with \((A,\Comult_A)\).  This duality is a contravariant functor on the bicharacter category.
\end{proposition}

\begin{proof}
  We check~\eqref{eq:linking_unitary_comult_C_hat} for~\(\hat{V}\) using~\eqref{eq:linking_unitary_comult_A} for~\(V\):
  \[
  (\Comult_A\otimes\Id_{\hat{C}})\flip(V^*)
  = \flip_{23}\flip_{12}((\Id_{\hat{C}}\otimes\Comult_A)V^*)
  = \flip_{23}\flip_{12}(V^*_{13}V^*_{12})
  = \flip(V^*)_{23}\cdot \flip(V^*)_{13},
  \]
  A similar computation yields~\eqref{eq:linking_unitary_comult_A} for~\(\hat{V}\).  A quantum group and its dual are canonically represented on the same Hilbert space, and the flip~\(\flip\) on operators is implemented by conjugating by~\(\Flip\).  Hence \(\hat{\mathbb{V}}\defeq\Flip\mathbb{V^*}\Flip\).

  Functoriality follows from the following computation:
  \[
  \widehat{\Linkunit{C}{B}_{13}}
  = \Flip_{13}\Linkunit{A}{B}_{23}(\Linkunit{C}{A}_{12})^*(\Linkunit{A}{B}_{23})^*\Linkunit{C}{A}_{12}\Flip_{13}
  = \widehat{\Linkunit{A}{B}_{12}}^*\widehat{\Linkunit{C}{A}_{23}}
  \widehat{\Linkunit{A}{B}_{12}}\widehat{\Linkunit{C}{A}_{23}}^*.
  \]
\end{proof}

The following result generalises \cite{Soltan-Woronowicz:Multiplicative_unitaries}*{Lemma 40} and is proved by the same idea.

\begin{proposition}
  \label{pro:opplinkunit_same_as_link_unit_in_mult_algebra}
  Let \(V\in\U\Mult(\hat{C}\otimes A)\) be a bicharacter.  Let \(R\) and~\(\tau\) denote the unitary antipodes and scaling groups of quantum groups.  Then
  \begin{align}
    \label{eq:unitary_antipode_linkunit_same_as_link_unit}
    (R_{\hat{C}}\otimes R_A)(V) &=V,\\
    \label{eq:scaling_group_linkunit}
    (\tau^{\hat{C}}_t\otimes\tau^A_t)(V) &=V \qquad
    \text{for all \(t\in\R\).}
  \end{align}
\end{proposition}

\begin{proof}
  Let \(\varphi\in\hat{C}_*\) and \(\psi\in A_*\) be entire analytic for \((\tau^{\hat{C}}_t)\) and~\((\tau^A_t)\), respectively.  Let \(\varphi_t \defeq \varphi\circ\tau^{\hat{C}}_t\) and \(\psi_t \defeq \psi\circ\tau^A_t\) for all \(t\in\R\).  Analytic continuation yields
  \[
  \varphi_{z+z'}=\varphi_z\circ\tau^{\hat{C}}_{z'},
  \qquad \text{and}\qquad
  \varphi_{z+z'}=\psi_z\circ\tau^A_{z'}
  \qquad\text{for all \(z,z'\in\C\).}
  \]
  Polar decomposition of the antipodes \(\kappa_{\hat{C}}\) and~\(\kappa_A\) (\cite{Woronowicz:Multiplicative_Unitaries_to_Quantum_grp}*{Theorem 1.5}) shows that
  \[
  \varphi_z\circ\kappa_{\hat{C}}=\varphi_{z+\nicefrac{\ima}{2}}\circ R_{\hat{C}},
  \qquad \text{and} \qquad
  \psi_z\circ\kappa_A=\psi_{z+\nicefrac{\ima}{2}}\circ R_A.
  \]
  Let~\(\bar{\kappa}_A\) be the closure of~\(\kappa_A\) with respect to the strict topology on \(\Mult(A)\).  Then \cite{Woronowicz:Multiplicative_Unitaries_to_Quantum_grp}*{Theorem 1.6(4)} yields
  \[
  \bar{\kappa}_A(\omega\otimes\Id)V=(\omega\otimes\Id)(V^*)
  \]
  for all \(\omega\in \hat{C}_*\).  Applying~\(\psi_z\) to both sides and using that~\(\omega\) is arbitrary, we get
  \[
  \bigl(\Id\otimes\psi_{z+\nicefrac{\ima}{2}}\circ R_A\bigr)V
  = (\Id\otimes\psi_z)(V^*).
  \]
  Interchanging the roles of \(A\) and~\(\hat{C}\) and replacing \(V\) by \(\Flip V^*\Flip\) and \(\psi\) by~\(\varphi\), we get
  \[
  \bigl(\varphi_{z+\nicefrac{\ima}{2}}\circ R_{\hat{C}}\otimes\Id\bigr)(V^*)
  = (\varphi_z\otimes\Id)V.
  \]
  Both formulas together yield
  \begin{multline}
    \label{eq:V_with_entire_analytic_functional}
    (\varphi_{z+\nicefrac{\ima}{2}}\otimes\psi_{z+\nicefrac{\ima}{2}})\circ
    (R_{\hat{C}}\otimes R_A)(V)
    = (\varphi_{z+\nicefrac{\ima}{2}}\circ R_{\hat{C}}\otimes\psi_{z+\nicefrac{\ima}{2}}\circ R_A)(V)\\
    = (\varphi_{z+\nicefrac{\ima}{2}}\circ R_{\hat{C}}\otimes\psi_z)(V^*)
    = \psi_z(\varphi_{z+\nicefrac{\ima}{2}}\circ R_{\hat{C}}\otimes\Id)(V^*)
    = (\varphi_z\otimes\psi_z)(V).
  \end{multline}
  Inserting \(\varphi\circ\kappa_{\hat{C}}\) and~\(\psi\circ\kappa_A\) into~\eqref{eq:V_with_entire_analytic_functional} instead of \(\varphi\) and~\(\psi\) yields
  \[
  (\varphi_{z+\ima}\otimes\psi_{z+\ima})(V)
  = (\varphi_{z+\nicefrac{\ima}{2}}\otimes\psi_{z+\nicefrac{\ima}{2}})
  \circ (R_{\hat{C}}\otimes R_A)(V)
  = (\varphi_z\otimes\psi_z)(V).
  \]
  This shows that \((\varphi_z\otimes\psi_z)(V)\) is a periodic function of period~\(\ima\).  Being bounded as well, Liouville's Theorem shows that it is constant, that is,
  \begin{align}
    \label{eq:V_with_const_entire_analytic_functional}
    (\varphi_z\otimes\psi_z)(V) &= (\varphi\otimes\psi)(V)
  \end{align}
  for all \(z\in\C\).  Putting \(z=-\nicefrac{\ima}{2}\) in~\eqref{eq:V_with_entire_analytic_functional} and using~\eqref{eq:V_with_const_entire_analytic_functional} yields
  \[
  (\varphi\otimes\psi)\circ (R_{\hat{C}}\otimes R_A)(V)=(\varphi\otimes\psi)(V).
  \]
  This proves \((R_{\hat{C}}\otimes R_A)(V)=V\).  Finally, \eqref{eq:V_with_const_entire_analytic_functional} also yields \((\tau^{\hat{C}}_t\otimes\tau^A_t)(V)=V\) for all \(t\in\R\).
\end{proof}

Besides taking duals, we may also take the opposite or coopposite of a quantum group, where we change the order of multiplication or comultiplication.  We remark without proof that these constructions are (covariant) functors on the bicharacter category.  The opposite-coopposite of a quantum group is isomorphic to the original quantum group via the unitary antipode because it is an antihomomorphism both for the algebra and the coalgebra structure.  The first part of Proposition~\ref{pro:opplinkunit_same_as_link_unit_in_mult_algebra} shows that this isomorphism acts identically on bicharacters.

%  \begin{remark}
%   The opposite, coopposite, and dual quantum groups are indeed quantum groups in our sense, that is, they may be obtained from a modular multiplicative unitary.  This is well-known for the dual quantum group, which is obtained from the multiplicative unitary \(\Flip\Multunit^*\Flip\in\U(\Hils\otimes\Hils)\) if \(\Multunit\in\U(\Hils\otimes\Hils)\) gives rise to~\((C,\Comult_C)\) and \(\Flip\in\U(\Hils\otimes\Hils)\) is the flip of the tensor factors.
%
%   The multiplicative unitary for the opposite quantum group acts on~\(\conj{\Hils}\otimes\conj{\Hils}\) for the complex-conjugate Hilbert space~\(\conj{\Hils}\).  An operator~\(w\) on~\(\Hils\) induces a transpose operator~\(w^\transpose\) on~\(\conj{\Hils}\) by \(w(\conj{\xi}) \defeq \conj{w^*\xi}\) for all \(\xi\in\Hils\).  The operator~\((\Multunit^*)^{\transpose\otimes\transpose}\) is a multiplicative unitary on~\(\conj{\Hils}\otimes\conj{\Hils}\) that gives rise to the opposite quantum group.
%
%   Finally, the coopposite quantum group is isomorphic to the opposite quantum group, so that it is generated by the same modular multiplicative unitary.
%
%   Proposition~\ref{pro:linking_opposites} shows that the bicharacters \(V^\opp\) and~\(V^\copp\) constructed out of a given bicharacter~\(V\) are the same when viewed as elements of \(\U\Mult(\hat{C}\otimes A)\).  However, since our identifications do not preserve the quantum group structure, these operators become different when we realise them concretely as unitary operators on a tensor product of two Hilbert spaces.
% \end{remark}

\begin{example}
  \label{ex:opposite_to_coopposite_from_definition_of_modularity}
  We give an interesting example of a concrete bicharacter from the definition of modular multiplicative unitaries.

  Let \((C,\Comult)\) be a quantum group generated by a modular multiplicative unitary~\(\Multunit\) on \(\Hils\otimes\Hils\).  The opposite quantum group \((C^\opp,\Comult)\) is generated by a modular multiplicative unitary acting on \(\conj{\Hils}\otimes\conj{\Hils}\) for the complex-conjugate Hilbert space~\(\conj{\Hils}\).  An operator~\(w\) on~\(\Hils\) induces a transpose operator~\(w^\transpose\) on~\(\conj{\Hils}\) by \(w^{\transpose}(\conj{\xi}) \defeq \conj{w^*\xi}\) for all \(\xi\in\Hils\).  The unitary operator~\((\Multunit^*)^{\transpose\otimes\transpose}\) on~\(\conj{\Hils}\otimes\conj{\Hils}\) is multiplicative and gives rise to the quantum group \((\bar{C},\bar{\Comult})\) where
  \begin{equation}
    \label{eq:opposite_quantum_group}
    \bar{C}=\{c^\transpose : c\in C\}
    \quad\text{and}\quad
    \bar{\Comult}(c^\transpose)=(\Comult(c))^{\transpose\otimes\transpose}.
  \end{equation}
  The quantum group \((\bar{C},\bar{\Comult})\) is isomorphic to \((C^\opp,\Comult)\).  Thus the dual \((\hat{\bar{C}},\hat{\bar{\Comult}})\) is isomorphic to the dual of the opposite quantum group \((\widehat{C^\opp},\hat{\Comult})\), where
  \begin{equation}
    \label{eq:dual_of_opposite_quantum_group}
    \hat{\bar{C}}=\{\hat{c}^\transpose : \hat{c}\in \hat{C}\}
    \quad\text{and}\quad
    \hat{\bar{\Comult}}(\hat{c}^\transpose)=(\hat{\Comult}(\hat{c}))^{\transpose\otimes\transpose}.
  \end{equation}

  Recall the operator \(\widetilde{\Multunit}\) from the definition of a modular multiplicative unitary, \cite{Soltan-Woronowicz:Remark_manageable}*{Definition 2.1}.  It satisfies \(\widetilde{\Multunit}^{*}=\Multunit^{\transpose\otimes R}\) by Theorem~\cite{Soltan-Woronowicz:Remark_manageable}*{Theorem 2.3(6(ii))}.  Hence~\eqref{eq:dual_of_opposite_quantum_group} yields \(\widetilde{\multunit}\in\U\Mult(\hat{\bar{C}}\otimes C)\).  We compute
  \begin{align*}
    (\hat{\bar{\Comult}}\otimes\Id_{C})\widetilde{\multunit}^{*}
    &=((\hat{\Comult}\otimes\Id_{C})\multunit)^{\transpose\otimes\transpose\otimes R}
    =(\multunit_{23}\multunit_{13})^{\transpose\otimes\transpose\otimes R}
    =\widetilde{\multunit}_{13}^{*}\widetilde{\multunit}_{23}^{*}.
    \\
    (\Id_{\hat{\bar{C}}}\otimes\flip\circ\Comult)\widetilde{\multunit}^{*}
    &=((\Id_{\hat{C}}\otimes \flip\circ\Comult)\multunit)^{\transpose\otimes R\otimes R}
    =(\multunit_{12}\multunit_{13})^{\transpose\otimes R\otimes R}
    =\widetilde{\multunit}_{13}^{*}\widetilde{\multunit}_{12}^{*}.
  \end{align*}
  Now Lemma~\ref{lemm:equiv_criterion_of_linkunit} shows that \(\widetilde{\multunit}\) is a bicharacter from \((\bar{C},\bar{\Comult})\) to \((C,\Comult^\opp)\).
\end{example}

\section{Passage to universal quantum groups}
\label{sec:universal}

In this section we show that our quantum group homomorphisms are equivalent to Hopf \(^*\)\nb-homomorphisms between the associated universal quantum groups, which were previously suggested as a suitable notion of quantum group homomorphism.  Moreover, on the way, we shall show that every \emph{reduced} bicharacter admits a unique bi-lift to a universal bicharacter.  Thus modular (or manageable) multiplicative unitaries are \emph{basic} in the sense of \cite{Ng:Morph_of_Mult_unit}*{Definition 2.3}.

Let \((C,\Comult_C)\) be a quantum group in the sense of~\cite{Soltan-Woronowicz:Multiplicative_unitaries}.  The associated universal quantum group \((C^\univ,\Comult_{C^\univ})\), also introduced in~\cite{Soltan-Woronowicz:Multiplicative_unitaries}, is a C*\nb-bialgebra, that is, a C*\nb-algebra equipped with a coassociative comultiplication.  While its structure is similar to that of a locally compact quantum group, it is usually not generated by a modular multiplicative unitary.  Thus the theory above does not apply to it.

\begin{definition}
  \label{def:left_corepresentation}
  A left \emph{corepresentation} of \((\hat{C},\Comult_{\hat{C}})\) on a C*\nb-algebra~\(D\) is a unitary multiplier \(V\in\U\Mult(\hat{C}\otimes D)\) that satisfies \((\Comult_{\hat{C}}\otimes \Id_D)(V) = V_{23} V_{13}\).
\end{definition}

The universal dual carries a left corepresentation \(\maxcorep\in\U\Mult(\hat{C}\otimes C^\univ)\) of~\(\hat{C}\) that is universal in the following sense: for any left corepresentation \(U\in\U\Mult(\hat{C}\otimes D)\) there is a unique morphism \(\varphi\colon C^\univ\to D\) with \(U = (\Id_{\hat{C}} \otimes \varphi)(\maxcorep)\).  This universal property characterises the pair \((C^\univ,\maxcorep)\) uniquely up to isomorphism.

The comultiplication on~\(C^\univ\) is defined so that \(\Id_{\hat{C}}\otimes \Comult_{C^\univ}\) maps~\(\maxcorep\) to the left corepresentation~\(\maxcorep_{12}\maxcorep_{13}\).  Thus~\(\maxcorep\) is a bicharacter and may be interpreted as a quantum group homomorphism from~\(C\) to~\(C^\univ\).  This is, however, not literally true because~\(C^\univ\) is not a quantum group in our sense.

\begin{proposition}
  \label{pro:qg_universal}
  Let \((A,\Comult_A)\) be a C*\nb-bialgebra.  Bicharacters in \(\U\Mult(\hat{C}\otimes A)\) correspond bijectively to Hopf \(^*\)\nb-homomorphisms from~\((C^\univ,\Comult_{C^\univ})\) to~\((A,\Comult_A)\).
\end{proposition}

\begin{proof}
  A Hopf \(^*\)\nb-homomorphism \(\varphi\colon C^\univ \to A\) is also a morphism from~\(C^\univ\) to~\(A\) and thus corresponds to a left corepresentation \(V\in\U\Mult(\hat{C}\otimes A)\), which is determined by the condition \((\Id_{\hat{C}}\otimes\varphi)(\maxcorep) = V\).  The Hopf \(^*\)\nb-homomorphisms \(\Comult_A\circ \varphi\colon C^\univ \to A\otimes A\) and \((\varphi\otimes\varphi)\circ\Comult_{C^\univ}\colon C^\univ \to A\otimes A\) correspond to the left corepresentations \((\Id_{\hat{C}}\otimes \Comult_A)(V)\) and~\(V_{12}V_{13}\), that is, \(\Id_{\hat{C}}\otimes (\Comult_A\circ \varphi)(\maxcorep) = (\Id_{\hat{C}}\otimes \Comult_A)(V)\) and \((\Id_{\hat{C}} \otimes (\varphi\otimes\varphi)\circ\Comult_{C^\univ})(\maxcorep) = V_{12}V_{13}\) because~\(\maxcorep\) is a bicharacter.  Thus a morphism \(\varphi\colon C^\univ \to A\) is a Hopf \(^*\)\nb-homomorphism if and only if the corepresentation~\(V\) also satisfies \((\Id_{\hat{C}}\otimes \Comult_A)(V) = V_{12} V_{13}\).  That is, \(V\) is a bicharacter.
\end{proof}

\begin{corollary}
  \label{cor:dual_strong_morph}
  Any Hopf \(^*\)\nb-homomorphism \(\varphi\colon C^\univ \to A\) induces a dual Hopf \(^*\)\nb-homomorphism \(\hat{\varphi}\colon \hat{A}^\univ \to \hat{C}\).
\end{corollary}

\begin{proof}
  By Proposition~\ref{pro:qg_universal}, a Hopf \(^*\)\nb-homomorphism \(\varphi\colon C^\univ \to A\) corresponds to a bicharacter~\(V\) in \(\U\Mult(\hat{C}\otimes A)\).  By Proposition~\ref{pro:dual_bicharacter}, \(\sigma(V^*)\) is a bicharacter from~\(\hat{A}\) to~\(\hat{C}\), which yields a Hopf \(^*\)\nb-homomorphism \(\varphi\colon \hat{A}^\univ \to \hat{C}\) by Proposition~\ref{pro:qg_universal}.
\end{proof}

We are going to show that Hopf \(^*\)\nb-homomorphisms from \((C^\univ,\Comult_{C^\univ})\) to \((A,\Comult_A)\) lift uniquely to Hopf \(^*\)\nb-homomorphisms from \((C^\univ,\Comult_{C^\univ})\) to \((A^\univ,\Comult_{A^\univ})\).  Together with Proposition~\ref{pro:qg_universal}, this yields a bijection between homomorphisms of quantum groups in our sense and Hopf \(^*\)\nb-homomorphisms between the associated universal quantum groups.
The main ingredient is the universal bicharacter \(\unibich\in\U\Mult(\hat{C}^\univ\otimes C^\univ)\).  For quantum groups with Haar weights, it is constructed in \cite{Kustermans:LCQG_universal}*{Proposition 6.4}.  First we carry this construction over to the setting of~\cite{Soltan-Woronowicz:Multiplicative_unitaries}.

The bicharacter~\(\multunit\) of~\(C\) is also a left corepresentation.  Hence the universal property yields a reducing \Star{}homomorphism \(\Lambda\colon C^\univ\to C\) with
\begin{equation}
  \label{eq:maxcorep_to_multunit}
  (\Id_{\hat{C}}\otimes\Lambda)(\maxcorep) = \multunit.
\end{equation}
The constructions above for the dual of~\(C\) yields a maximal left corepresentation \(\dumaxcorep\in\U\Mult(\hat{C}^\univ\otimes C)\) of~\(C\) and a reducing \Star{}homomorphism \(\hat{\Lambda}\colon \hat{C}^\univ\to \hat{C}\) with
\begin{equation}
  \label{eq:dumaxcorep_to_multunit}
  (\Lambda\otimes\Id_C)(\dumaxcorep) = \multunit.
\end{equation}
We want to find \(\unibich\in\U\Mult(\hat{C}^\univ\otimes C^\univ)\) with \((\hat{\Lambda}\otimes\Id_{C^\univ})(\unibich) = \maxcorep\) and \((\Id_{\hat{C}^\univ}\otimes\Lambda)(\unibich) = \dumaxcorep\).

Using~\eqref{eq:Delta_via_W}, we may rewrite the fact that~\(\dumaxcorep\) is a corepresentation in the second variable as a pentagon equation
\begin{equation}
  \label{eq:VVW_pentagon1}
  \multunit_{23} \dumaxcorep_{12} =
  \dumaxcorep_{12} \dumaxcorep_{13} \multunit_{23}
  \qquad \text{in \(\U\Mult(\hat{C}^\univ \otimes \Comp(\Hils_C) \otimes C)\).}
\end{equation}
Similarly, using~\eqref{eq:hat_Delta_via_W} and that~\(\maxcorep\) is a corepresentation in the first variable, we get the pentagon equation
\begin{equation}
  \label{eq:VVW_pentagon2}
  \maxcorep_{23} \multunit_{12}
  = \multunit_{12} \maxcorep_{13} \maxcorep_{23}
  \qquad \text{in \(\U\Mult(\hat{C} \otimes \Comp(\Hils_C) \otimes C^\univ)\).}
\end{equation}
In both cases, we represent the second tensor factors \(C\) and~\(\hat{C}\) (faithfully) on~\(\Hils_C\) to make sense of the pentagon equation.  We may now characterise~\(\unibich\) by a variant of the pentagon equation as in \cite{Kustermans:LCQG_universal}*{Proposition 6.4}.

\begin{proposition}
  \label{pro:universal_corepresentation}
  There is a unique \(\unibich\in\U\Mult(\hat{C}^\univ\otimes C^\univ)\) such that
  \[
  \maxcorep_{23}\dumaxcorep_{12} = \dumaxcorep_{12} \unibich_{13}\maxcorep_{23}
  \qquad \text{in \(\U\Mult(\hat{C}^\univ \otimes \Comp(\Hils_C) \otimes C^\univ)\).}
  \]
  Moreover, this~\(\unibich\) is a bicharacter, and it satisfies
  \begin{align}
    \label{eq:univ_corep_of_C_from_U}
    (\Id_{\hat{C}^\univ}\otimes\Lambda)\unibich &= \dumaxcorep,\\
    \label{eq:univ_corep_of_hat_C_from_U}
    (\hat{\Lambda}\otimes\Id_{C^\univ})\unibich &= \maxcorep,\\
    \label{eq:reduced_bicharacter_from_U}
    (\hat{\Lambda}\otimes\Lambda)\unibich &= \multunit.
  \end{align}
\end{proposition}

\begin{proof}
  Let \(\unibich' \defeq \dumaxcorep_{12}^* \maxcorep_{23}\dumaxcorep_{12}\maxcorep_{23}^*\).  First, we show \(\unibich'\in \Mult(\hat{C}^\univ\otimes 1\otimes C^\univ)\), that is, \(\unibich' = \unibich_{13}\) for some \(\unibich\in\U\Mult(\hat{C}^\univ\otimes C^\univ)\).  Obviously, this unitary is the unique solution of our problem.  Then we establish that~\(\unibich\) is a bicharacter.

  The first step follows once again from Theorem~\ref{the:co-invariant_type_operators}.  We compute
  \begin{align*}
    \Multunit_{23} \unibich'_{124} \Multunit_{23}^*
    &= \multunit_{23}\dumaxcorep_{12}^*\maxcorep_{24}\dumaxcorep_{12}\maxcorep_{24}^* \multunit_{23}^*\\
    &= \dumaxcorep_{13}^*\dumaxcorep_{12}^*\multunit_{23}\maxcorep_{24}\dumaxcorep_{12} \maxcorep_{24}^*\multunit_{23}^*\\
    &= \dumaxcorep_{13}^*\dumaxcorep_{12}^*\maxcorep_{34}\multunit_{23}\maxcorep_{34}^* \dumaxcorep_{12}\maxcorep_{34}\multunit_{23}^*\maxcorep_{34}^*\\
    &= \dumaxcorep_{13}^*\dumaxcorep_{12}^*\maxcorep_{34}\multunit_{23}\dumaxcorep_{12} \multunit_{23}^*\maxcorep_{34}^*\\
    &= \dumaxcorep_{13}^*\dumaxcorep_{12}^*\maxcorep_{34}\dumaxcorep_{12}\dumaxcorep_{13} \maxcorep_{34}^*
    = \dumaxcorep_{13}^*\maxcorep_{34}\dumaxcorep_{13}\maxcorep_{34}^*
    = \unibich'_{134};
  \end{align*}
  the first step is the definition of~\(\unibich'\); the second step uses~\eqref{eq:VVW_pentagon1}; the third step uses~\eqref{eq:VVW_pentagon2} twice; the fourth step uses that \(\maxcorep_{34}^*\) and~\(\multunit_{12}\) commute; the fifth step again uses~\eqref{eq:VVW_pentagon1}; and the sixth step follows because \(\maxcorep_{34}\) and~\(\dumaxcorep_{12}\) commute.  Now Theorem~\ref{the:co-invariant_type_operators} yields \(\unibich'\in\U\Mult(\hat{C}^\univ\otimes 1\otimes C^\univ)\), so that~\(\unibich\) exists.

  Now we show that~\(\unibich\) is a corepresentation in the second variable:
  \begin{multline*}
    (\Id_{\hat{C}^\univ}\otimes\Id_C\otimes\Comult_{C^\univ})\dumaxcorep_{12}^*\maxcorep_{23}
    \dumaxcorep_{12}\maxcorep_{23}^*
    = \dumaxcorep_{12}^*\maxcorep_{23}\maxcorep_{24}\dumaxcorep_{12} \maxcorep_{24}^*\maxcorep_{23}^*
    \\= \unibich_{13}\maxcorep_{23}\dumaxcorep_{12}^*\maxcorep_{24}\dumaxcorep_{12} \maxcorep_{24}^*\maxcorep_{23}^*
    = \unibich_{13}\maxcorep_{23}\unibich_{14}\maxcorep_{23}^*
    = \unibich_{13}\unibich_{14}.
  \end{multline*}
  A similar computation works in the first variable.  Thus~\(\unibich\) is a bicharacter.

  The following computation yields~\eqref{eq:univ_corep_of_C_from_U}:
  \begin{multline*}
    (\Id_{\hat{C^\univ}}\otimes\Id_C\otimes\Lambda)\unibich_{13}
    = (\Id_{\hat{C}^\univ}\otimes\Id_C\otimes\Lambda)\dumaxcorep_{12}^*\maxcorep_{23}\dumaxcorep_{12}\maxcorep_{23}^*\\
    = \dumaxcorep_{12}^*\multunit_{23}\dumaxcorep_{12}\multunit_{23}^*
    = \dumaxcorep_{12}^*\dumaxcorep_{12}\dumaxcorep_{13}
    = \dumaxcorep_{13}.
  \end{multline*}
  A similar computation yields~\eqref{eq:univ_corep_of_hat_C_from_U}.  Then~\eqref{eq:reduced_bicharacter_from_U} follows from~\eqref{eq:maxcorep_to_multunit} or~\eqref{eq:dumaxcorep_to_multunit}.
\end{proof}

\begin{definition}
  \label{def:univ_co_representation}
  The unitary multiplier~\(\unibich\) in Proposition~\ref{pro:universal_corepresentation} is called the \emph{universal bicharacter} of \((C,\Comult_C)\).
\end{definition}

\begin{lemma}
  \label{lemm:equal_corep_of_univ_object_slice_by_reducing_morph}
  Let \(X, Y\in\U\Mult(C\otimes A^\univ)\) be corepresentations in the second variable.  Let \(\pi\colon A^\univ\to A\) be the reducing \Star{}homomorphism.  If \((\Id_C\otimes\pi)X=(\Id_C\otimes\pi)Y\), then \(X=Y\).  A similar statement holds in the first variable.
\end{lemma}

\begin{proof}
  Copy the proof of \cite{Kustermans:LCQG_universal}*{Result 6.1}.
\end{proof}

\begin{proposition}
  \label{pro:unique_lift_bicharacter}
  A bicharacter in \(\U\Mult(\hat{C}\otimes A)\) lifts uniquely to a bicharacter in \(\U\Mult(\hat{C}^\univ\otimes A^\univ)\) and hence to bicharacters in \(\U\Mult(\hat{C}\otimes A^\univ)\) and \(\U\Mult(\hat{C}^\univ\otimes A)\).
\end{proposition}

\begin{proof}
  These liftings are unique by Lemma~\ref{lemm:equal_corep_of_univ_object_slice_by_reducing_morph}.  It remains to prove existence.  Let \(V\in\U\Mult(\hat{C}\otimes A)\) be a bicharacter.  By Proposition~\ref{pro:qg_universal}, it corresponds to a Hopf \(^*\)\nb-homomorphism \(\varphi\colon C^\univ\to A\).  Let \(\unibich^C\in\U\Mult(\hat{C}^\univ\otimes C^\univ)\) be the universal bicharacter.  Then \(V'\defeq (\Id\otimes\varphi)\unibich^C\in\U\Mult(\hat{C}^\univ\otimes A)\) is a bicharacter that lifts~\(V\).  Now \(\sigma(V')^*\in\U\Mult(A\otimes\hat{C}^\univ)\) is again a bicharacter (see Proposition~\ref{pro:dual_bicharacter}).  Repeating the above step we lift it to a bicharacter~\(V''\) in \(\U\Mult(A^\univ\otimes\hat{C}^\univ)\).  Then \(\sigma(V'')^*\) is the desired lifting of~\(V\).
\end{proof}

Recall that bicharacters form a category and that duality is a functor on this category.  Hopf \(^*\)\nb-homomorphisms \(A^\univ\to C^\univ\) also form the arrows of a category.

\begin{theorem}
  \label{the:linking_universal}
  There is an isomorphism between the categories of locally compact quantum groups with bicharacters from~\(C\) to~\(A\) and with Hopf \(^*\)\nb-homomorphisms \(C^\univ\to A^\univ\) as morphisms \(C\to A\), respectively.  The bicharacter associated to a Hopf \(^*\)\nb-homomorphism \(\varphi\colon C^\univ\to A^\univ\) is \((\Lambda_{\hat{C}}\otimes\Lambda_A\varphi)(\unibich^C) \in \U\Mult(\hat{C}\otimes A)\).

  Furthermore, the duality on the level of bicharacters corresponds to the duality \(\varphi\mapsto\hat{\varphi}\) on Hopf \(^*\)\nb-homomorphisms, where \(\hat{\varphi}\colon \hat{A}^\univ\to\hat{C}^\univ\) is the unique Hopf \(^*\)\nb-homomorphism with \((\hat{\varphi}\otimes\Id_{A^\univ})(\unibich^A) = (\Id_{\hat{C}^\univ}\otimes\varphi)(\unibich^C)\).
\end{theorem}

\begin{proof}
  Propositions \ref{pro:qg_universal} and~\ref{pro:unique_lift_bicharacter} yield bijections from Hopf \(^*\)\nb-homomorphisms \(C^\univ\to A^\univ\) to bicharacters from~\(C\) to~\(A^\univ\) and on to bicharacters from~\(C\) to~\(A\).  We must check that this bijection preserves the compositions and the duality.  We first turn to the duality because we need this to establish the compatibility with compositions.

  Let \(\varphi\colon C^\univ\to A^\univ\) be a Hopf \(^*\)\nb-homomorphism.  Let \(V\defeq (\Lambda_{\hat{C}}\otimes \Lambda_A\varphi)(\unibich^C) \in \U\Mult(\hat{C}\otimes A)\) be the associated bicharacter.  The duality on the level of bicharacters yields the bicharacter \(\sigma(V^*) \in \U\Mult(A\otimes\hat{C})\) from~\(\hat{A}\) to~\(\hat{C}\).  This corresponds to a unique Hopf \(^*\)\nb-homomorphism \(\hat\varphi\colon \hat{A}^\univ\to\hat{C}^\univ\) with \((\Lambda_A\otimes\Lambda_{\hat{C}}\hat\varphi)(\unibich^{\hat{A}}) = \sigma(V)^*\).  Now we use \(\unibich^{\hat{A}} = \sigma(\unibich^A)^*\) to rewrite this as
  \[
  (\Lambda_{\hat{C}}\otimes\Lambda_A\varphi)(\unibich^C)
  = (\Lambda_{\hat{C}}\hat\varphi\otimes\Lambda_A)(\unibich^A).
  \]
  Both \((\Id\otimes\varphi)(\unibich^C)\) and \((\hat\varphi\otimes\Id)(\unibich^A)\) are bicharacters.  Applying Lemma~\ref{lemm:equal_corep_of_univ_object_slice_by_reducing_morph} to both tensor factors, we get first \((\Id_{\hat{C}^\univ}\otimes\Lambda_A\varphi)(\unibich^C) = (\hat\varphi\otimes\Lambda_A)(\unibich^A)\) and then \((\Id_{\hat{C}^\univ}\otimes\varphi)(\unibich^C) = (\hat\varphi\otimes\Id_{A^\univ})(\unibich^A)\).  This yields the asserted description of duality.

  Now let \(\varphi\colon C^\univ\to A^\univ\) and \(\psi\colon A^\univ\to B^\univ\) be Hopf \(^*\)\nb-homomorphisms and let \(\Linkunit{C}{A}\in\U\Mult(\hat{C}\otimes A)\) and \(\Linkunit{A}{B}\in\U\Mult(\hat{A}\otimes B)\) be the corresponding bicharacters,
  \begin{align*}
    \Linkunit{A}{B}
    &= (\Lambda_{\hat{A}}\otimes\Lambda_B\psi)\unibich^A
    = (\Id_{\hat{A}^\univ}\otimes\Lambda_B\psi)\maxcorep^A,
    \\
    \Linkunit{C}{A}
    &= (\Lambda_{\hat{C}}\hat{\varphi}\otimes\Lambda_A)\unibich^A
    = (\Lambda_{\hat{C}}\hat{\varphi}\otimes\Id_{A^\univ})\maxcorep^A,
  \end{align*}
  where we use the dual quantum group homomorphism \(\hat{\varphi}\colon \hat{A}^\univ\to \hat{C}^\univ\).  Now
  \begin{align*}
    (\Linkunit{A}{B} * \Linkunit{C}{A})_{13} &=
    (\Linkunit{C}{A}_{12})^* \Linkunit{A}{B}_{23}\Linkunit{C}{A}_{12}(\Linkunit{A}{B}_{23})^*
    \\&= (\Lambda_A\hat{\varphi}\otimes\Id_{\Bound(\Hils_A)}\otimes\Lambda_B\psi)
    ((\dumaxcorep^A_{12})^*\maxcorep^A_{23}\dumaxcorep^A_{12}(\maxcorep^A_{23})^*)
    \\&= (\Lambda_A\hat{\varphi}\otimes\Id_{\Bound(\Hils_A)}\otimes\Lambda_B\psi)
    (\unibich^A_{13})
  \end{align*}
  by Proposition~\ref{pro:universal_corepresentation}.  Thus
  \begin{multline*}
    \Linkunit{A}{B} * \Linkunit{C}{A}
    = (\Lambda_A\hat{\varphi}\otimes\Lambda_B\psi)(\unibich^A)
    = (\Lambda_A\otimes\Lambda_B\psi)\circ(\hat{\varphi}\otimes \Id_{A^\univ})(\unibich^A)
    \\= (\Lambda_A\otimes\Lambda_B\psi)\circ(\Id_{\hat{C}^\univ} \otimes \varphi) (\unibich^C)
    = (\Lambda_A\otimes\Lambda_B(\psi\circ\phi)) (\unibich^C).
  \end{multline*}
  Hence \(\Linkunit{A}{B} * \Linkunit{C}{A}\) is the bicharacter associated to \(\psi\circ\phi\).  Thus our bijection is compatible with compositions.
\end{proof}

\section{Right and left coactions}
\label{sec:right_left_action}

\begin{definition}
  \label{def:right_quantum_group_homomorphism}
  A \emph{right quantum group homomorphism} from \((C,\Comult_C)\) to \((A,\Comult_A)\) is a morphism \(\Delta_R\colon C\to C\otimes A\) for which the two diagrams in~\eqref{eq:intro_right_homomorphism} commute.
\end{definition}

The second diagram in~\eqref{eq:intro_right_homomorphism} means that~\(\Delta_R\) is an \(A\)\nb-comodule structure on~\(C\).

\begin{example}
  \label{exa:qg_strong_homomorphism}
  A Hopf \(^*\)\nb-homomorphism \(\varphi\colon C\to \Mult(A)\) yields a right quantum group homomorphism by \(\Delta_R \defeq (\Id_C\otimes\varphi)\Comult_C\).
\end{example}

Let \(\multunit\in \U\Mult(\hat{C}\otimes C)\) denote the reduced bicharacter.

\begin{theorem}
  \label{the:linking_right_homomorphism}
  For any right quantum group homomorphism \(\Delta_R\colon C\to C\otimes A\), there is a unique unitary \(V\in\U\Mult(\hat{C}\otimes A)\) with
  \begin{equation}
    \label{eq:def_V_via_right_homomorphism}
    (\Id_{\hat{C}} \otimes \Delta_R)(\multunit) = \multunit_{12}V_{13}.
  \end{equation}
  This unitary is a bicharacter.

  Conversely, let~\(V\) be a bicharacter from~\(C\) to~\(A\), and let \(\mathbb{V}\in\U(\Hils_C\otimes\Hils_A)\) be the corresponding concrete bicharacter.  Then
  \begin{equation}
    \label{eq:Delta_R_via_V}
    \Delta_R(x) \defeq \mathbb{V}(x\otimes1)\mathbb{V}^*\qquad
    \text{for all \(x\in C\)}
  \end{equation}
  defines a right quantum group homomorphism from~\(C\) to~\(A\).

  These two maps between bicharacters and right quantum group homomorphisms are inverse to each other.
\end{theorem}

\begin{proof}
  First we check that \(\tilde{V} \defeq \multunit_{12}^*\cdot (\Id_{\hat{C}}\otimes\Delta_R)(\multunit)\) belongs to \(\U\Mult(\hat{C}\otimes 1\otimes A)\), that is, \(\tilde{V}=V_{13}\) for some \(V\in\U\Mult(\hat{C}\otimes A)\).  This is the unique~\(V\) that verifies~\eqref{eq:def_V_via_right_homomorphism}.  We compute
  \begin{align*}
    \Multunit^C_{23} \tilde{V}_{124} (\Multunit^C_{23})^*
    &= \Multunit^C_{23} (\Multunit^C_{12})^* (\Multunit^C_{23})^* \cdot
    \Multunit^C_{23} (\Id_{\hat{C}}\otimes\Delta_R)(\multunit)_{124} (\Multunit^C_{23})^*
    \\&= (\Multunit^C_{13})^* (\Multunit^C_{12})^* \cdot
    (\Id_{\hat{C}}\otimes\Comult_C\otimes\Id_A) (\Id_{\hat{C}}\otimes\Delta_R)(\multunit)
    \\&= (\Multunit^C_{13})^* (\Multunit^C_{12})^* \cdot
    (\Id_{\hat{C}}\otimes(\Id_C\otimes\Delta_R)\Comult_C)\multunit
    \\&= (\Multunit^C_{13})^* (\Multunit^C_{12})^* \cdot
    (\Id_{\hat{C}}\otimes\Id_C\otimes\Delta_R)(\multunit_{12}\multunit_{13})
    \\&= (\Multunit^C_{13})^*
    \bigl((\Id_{\hat{C}}\otimes\Delta_R)\multunit\bigr)_{134};
  \end{align*}
  the first equality is the definition of~\(\tilde{V}\), the second one uses \eqref{eq:pentagon} and~\eqref{eq:Delta_via_W}, the third one~\eqref{eq:intro_right_homomorphism}, the fourth one uses~\eqref{eq:Delta_W}, and the last one is trivial.  Now Theorem~\ref{the:co-invariant_type_operators} yields \(\tilde{V} = V_{13}\) for some \(V\in\U\Mult(\hat{C}\otimes A)\).

  Next we verify that~\(V\) is a bicharacter.  We check~\eqref{eq:linking_unitary_comult_C_hat}:
  \begin{align*}
    \bigl((\Comult_{\hat{C}}\otimes\Id_A)V\bigr)_{124}
    &= (\Comult_{\hat{C}}\otimes\Id_C\otimes\Id_A) \bigl(\multunit_{12}^*\cdot
    (\Id_{\hat{C}}\otimes\Delta_R)(\multunit)\bigr)
    \\& = ((\Comult_{\hat{C}}\otimes\Id_C)\multunit^*)_{123}\cdot
    (\Id_{\hat{C}}\otimes\Id_{\hat{C}}\otimes\Delta_R)(\Comult_{\hat{C}}\otimes\Id_C)(\multunit)
    \\&= (\multunit_{23}\multunit_{13})^*
    (\Id_{\hat{C}}\otimes\Id_{\hat{C}}\otimes\Delta_R)(\multunit_{23}\multunit_{13})
    \\&= \multunit_{13}^* \multunit_{23}^* \multunit_{23}V_{24}\multunit_{13}V_{14}
    = V_{24}V_{14};
  \end{align*}
  the first two equalities use~\eqref{eq:def_V_via_right_homomorphism} and that~\(\Comult_{\hat{C}}\) is a \Star{}homomorphism; the third equality uses~\eqref{eq:hat_Delta_W}; the fourth one uses~\eqref{eq:def_V_via_right_homomorphism} again; and the final step uses that \(\multunit_{13}\) and~\(V_{24}\) commute.  The following computation yields~\eqref{eq:linking_unitary_comult_A}:
  \begin{align*}
    \bigl((\Id_{\hat{C}}\otimes\Comult_A)V\bigr)_{134}
    &= \multunit_{12}^* (\Id_{\hat{C}}\otimes\Id_C\otimes\Comult_A)
    (\Id_{\hat{C}}\otimes\Delta_R)\multunit
    \\&= \multunit_{12}^* (\Id_{\hat{C}}\otimes\Delta_R\otimes\Id_A)
    (\Id_{\hat{C}}\otimes\Delta_R)\multunit
    \\&= \multunit_{12}^* (\Id_{\hat{C}}\otimes\Delta_R\otimes\Id_A)(\multunit_{12}V_{14})
    = V_{13}V_{14};
  \end{align*}
  the first equality follows from~\eqref{eq:def_V_via_right_homomorphism}; the second one from~\eqref{eq:intro_right_homomorphism}; the third and fourth equalities from~\eqref{eq:def_V_via_right_homomorphism}.  Thus we have constructed a bicharacter~\(V\) from a right quantum group homomorphism.

  Conversely, let \(V\in\U\Mult(\hat{C}\otimes A)\) be a bicharacter.  We claim that~\eqref{eq:Delta_R_via_V} defines a morphism from~\(C\) to \(C\otimes A\).  Recall that slices of~\(\Multunit\) by linear functionals \(\omega\in\Bound(\Hils)_*\) generate a dense subspace of~\(C\).  On \(x\defeq (\omega\otimes\Id_\Hils)(\Multunit)\), we compute
  \[
  \Delta_R(x)
  = (\omega\otimes\Id_\Hils\otimes\Id_\Hils)
  (\mathbb{V}_{23}\Multunit_{12}\mathbb{V}^*_{23})
  = (\omega\otimes\Id_\Hils\otimes\Id_\Hils)
  (\Multunit_{12}\mathbb{V}_{13}),
  \]
  and this belongs to \(\Mult(C\otimes A)\).  Thus \(\Delta_R(C)\subseteq \Mult(C\otimes A)\).  It is clear from the definition that~\(\Delta_R\) is non-degenerate.

  We may also rewrite the above computation as \((\omega\otimes\Id_{C\otimes A}) \circ (\Id_{\hat{C}}\otimes\Delta_R)(\multunit) = (\omega\otimes\Id_{C\otimes A})(\multunit_{12}V_{13})\) for all \(\omega\in\Bound(\Hils)_*\).  Since~\(\omega\) is arbitrary, \eqref{eq:def_V_via_right_homomorphism} holds for~\(\Delta_R\) and our original bicharacter~\(V\).

  Now we use~\eqref{eq:def_V_via_right_homomorphism} to check that~\(\Delta_R\) is a right quantum group homomorphism.  The first diagram in~\eqref{eq:intro_right_homomorphism} amounts to
  \[
  (\Id_{\hat{C}}\otimes\Comult_C\otimes\Id_A)(\Id_{\hat{C}}\otimes\Delta_R)(\multunit)
  = (\Id_{\hat{C}}\otimes\Id_C\otimes\Delta_R)(\Id_{\hat{C}}\otimes\Comult_C)(\multunit)
  \]
  because slices of~\(\multunit\) generate~\(C\).  This follows from \eqref{eq:def_V_via_right_homomorphism} and~\eqref{eq:Delta_W}: both sides are equal to~\(\multunit_{12}\multunit_{13}V_{14}\).  Similarly, the second diagram in~\eqref{eq:intro_right_homomorphism} amounts to
  \[
  (\Id_{\hat{C}}\otimes\Id_C\otimes\Comult_A)(\Id_{\hat{C}}\otimes\Delta_R)(\multunit)=
  (\Id_{\hat{C}}\otimes\Delta_R\otimes\Id_A)(\Id_{\hat{C}}\otimes\Delta_R)(\multunit),
  \]
  which follows from \eqref{eq:def_V_via_right_homomorphism} and~\eqref{eq:linking_unitary_comult_A} because both sides are equal to~\(\multunit_{12}V_{13}V_{14}\).

  Thus a bicharacter~\(V\) yields a right quantum group homomorphism~\(\Delta_R\).  Since these are related by~\eqref{eq:def_V_via_right_homomorphism}, we get back the original bicharacter from this right quantum group homomorphism.  It only remains to check that, if we start with a right quantum group homomorphism~\(\Delta_R\), define a bicharacter by~\eqref{eq:def_V_via_right_homomorphism} and then a right quantum group homomorphism by~\eqref{eq:Delta_R_via_V}, we get back the original~\(\Delta_R\).  We may rewrite~\eqref{eq:linking_unitary_pentagon_C} as
  \[
  \mathbb{V}_{23}\Multunit_{12}\mathbb{V}_{23}^*
  = \Multunit_{12}\mathbb{V}_{13}
  = (\Id_{\hat{C}}\otimes\Delta_R)(\Multunit),
  \]
  using~\eqref{eq:def_V_via_right_homomorphism}.  This implies that the original~\(\Delta_R\) satisfies~\eqref{eq:Delta_R_via_V} because the slices of~\(\multunit\) by linear functionals on~\(\hat{C}\) span a dense subspace of~\(C\).
\end{proof}

\begin{definition}
  \label{def:left_quantum_group_homomorphism}
  A \emph{left quantum group homomorphism} from~\((C,\Comult_C)\) to~\((A,\Comult_A)\) is a morphism \(\Delta_L\colon C\to A\otimes C\) such that the following two diagrams commute:
  \[
  \xymatrix@C+2em{
    C \ar[r]^{\Delta_L}\ar[d]_{\Comult_C}&
    A\otimes C \ar[d]^{\Id_A\otimes\Comult_C}\\
    C\otimes C \ar[r]_{\Delta_L\otimes\Id_C}&
    A\otimes C\otimes C,
  }
  \qquad
  \xymatrix@C+2em{
    C \ar[r]^{\Delta_L}\ar[d]_{\Delta_L}&
    A\otimes C \ar[d]^{\Comult_A\otimes\Id_C}\\
    A\otimes C \ar[r]_{\Id_A\otimes\Delta_L}&
    A\otimes A\otimes C.
  }
  \]
\end{definition}

\begin{theorem}
  \label{the:linking_left_homomorphism}
  For any left quantum group homomorphism \(\Delta_L\colon C\to A\otimes C\), there is a unique unitary \(V\in\U\Mult(\hat{C}\otimes A)\) with
  \begin{equation}
    \label{eq:def_V_via_left_homomorphism}
    (\Id_{\hat{C}} \otimes \Delta_L)(\multunit) = V_{12}\multunit_{13}.
  \end{equation}
  This unitary is a bicharacter.

  Conversely, let~\(V\) be a bicharacter from~\(C\) to~\(A\), let \(\mathbb{V}\in\U(\Hils_C\otimes\Hils_A)\) be the corresponding concrete bicharacter, and define~\(\hat{\mathbb{V}}\) as in Proposition~\textup{\ref{pro:dual_bicharacter}}.  Let \(R_A\) and~\(R_C\) be the unitary antipodes of \(A\) and~\(C\).  Then \(\hat{\mathbb{V}}^*(1\otimes R_C(x))\hat{\mathbb{V}} \in \Mult(A\otimes C)\) for all \(x\in C\) and
  \begin{equation}
    \label{eq:Delta_L_via_V}
    \Delta_L(x) \defeq (R_A\otimes R_C)(\hat{\mathbb{V}}^{*}(1\otimes R_C(x))\hat{\mathbb{V}})
    \qquad
    \text{for all \(x\in C\)}
  \end{equation}
  is a left quantum group homomorphism from~\(C\) to~\(A\).

  These two maps between bicharacters and left quantum group homomorphisms are bijective and inverse to each other.
\end{theorem}

\begin{proof}
  As in the proof of Theorem~\ref{the:linking_right_homomorphism}, it may be shown that there is a unique~\(V\) satisfying~\eqref{eq:def_V_via_left_homomorphism} and that~\(\Delta_L\) is a well-defined left quantum group morphism \(C\to A\otimes C\).  The only point in the proof of Theorem~\ref{the:linking_right_homomorphism} that must be modified is to show that~\(\Delta_L\) given by~\eqref{eq:Delta_L_via_V} satisfies~\eqref{eq:def_V_via_left_homomorphism}.  We compute:
  \begin{align*}
    (\Id_{\hat{C}}\otimes\Delta_L)\multunit
    &= (\Id_{\hat{C}}\otimes\Delta_L)((R_{\hat{C}}\otimes R_C)\circ\multunit) \\
    &= (R_{\hat{C}}\otimes R_A\otimes R_{C})(\hat{\mathbb{V}}_{23}^{*}\Multunit_{13}\hat{\mathbb{V}}_{23})\\
    &= (R_{\hat{C}}\otimes R_A\otimes R_{C})\circ\flip_{23}(\mathbb{V}_{23}\Multunit_{12}\mathbb{V}_{23}^{*})\\
    &= (R_{\hat{C}}\otimes R_A\otimes R_{C})(\multunit_{13}V_{12})
    = V_{12}\multunit_{13};
  \end{align*}
  the first step uses Proposition~\ref{pro:opplinkunit_same_as_link_unit_in_mult_algebra} for~\(\multunit\), the second one uses~\eqref{eq:Delta_L_via_V}, the third one is trivial, the fourth one uses~\eqref{eq:linking_unitary_pentagon_C}, and the last one follows from Proposition~\ref{pro:opplinkunit_same_as_link_unit_in_mult_algebra} and the antimultiplicativity of~\(R_{\hat{C}}\).
\end{proof}

\begin{lemma}
  \label{lem:comp_left_right_homomorphism}
  Let \(\Delta_L\colon C\to A\otimes C\) and \(\Delta_R\colon C\to C\otimes B\) be a left and a right quantum group homomorphism.  Then the following diagram commutes:
  \[
   \begin{gathered}
   \label{eq:lef_right_homomorphism}
   \xymatrix@C+2em{
    C \ar[r]^{\Delta_L} \ar[d]_{\Delta_R} &
    A\otimes C \ar[d]^{\Id_A\otimes\Delta_R}\\
    C\otimes B \ar[r]_{\Delta_L\otimes\Id_B}& A\otimes C\otimes B.
  }
  \end{gathered}
  \]
  Furthermore, \(\Delta_L\) and~\(\Delta_R\) are associated to the same bicharacter \(V\in\U\Mult(\hat{C}\otimes A)\) if and only if the following diagram commutes:
  \begin{equation}
    \label{eq:comp_left_right_homomorphism}
    \begin{gathered}
      \xymatrix@C+2em{
        C \ar[r]^{\Comult_C} \ar[d]_{\Comult_C} &
        C\otimes C \ar[d]^{\Id_C\otimes\Delta_L}\\
        C\otimes C \ar[r]_{\Delta_R\otimes\Id_C}& C\otimes A\otimes C.
      }
    \end{gathered}
  \end{equation}
\end{lemma}

\begin{proof}
  Since slices of~\(\multunit^C\) span a dense subspace of~\(C\), \eqref{eq:lef_right_homomorphism} commutes if and only if
  \begin{equation}
    \label{eq:left_right_compatible}
    (\Id_{\hat{C}}\otimes\Id_A\otimes\Delta_R)
    (\Id_{\hat{C}}\otimes\Delta_L) (\multunit)
    = (\Id_{\hat{C}}\otimes\Delta_L\otimes\Id_B)
    (\Id_{\hat{C}}\otimes\Delta_R) (\multunit).
  \end{equation}
  Let \(V\) and~\(\tilde{V}\) be the bicharacters associated to \(\Delta_L\) and~\(\Delta_R\), respectively.  Equations \eqref{eq:def_V_via_right_homomorphism} and~\eqref{eq:def_V_via_left_homomorphism} imply that both sides of~\eqref{eq:left_right_compatible} are equal to~\(V_{12}\multunit_{13}\tilde{V}_{14}\).

  The diagram~\eqref{eq:comp_left_right_homomorphism} commutes if and only if
  \begin{equation}
    \label{eq:comp_left_right_compatible}
    (\Id_{\hat{C}}\otimes\Id_C\otimes\Delta_L)(\Id_{\hat{C}}\otimes\Comult_C)(\multunit)
    = (\Id_{\hat{C}}\otimes\Delta_R\otimes\Id_C)(\Id_{\hat{C}}\otimes\Comult_C)(\multunit)
  \end{equation}
  because slices of~\(\multunit\) span a dense subspace of~\(C\).  Using \eqref{eq:Delta_W}, \eqref{eq:def_V_via_left_homomorphism} and~\eqref{eq:def_V_via_right_homomorphism}, we may rewrite~\eqref{eq:comp_left_right_compatible} as
  \(\multunit_{12}\tilde{V}_{13}\multunit_{14}
  = \multunit_{12}V_{13}\multunit_{14}\).
  Thus~\eqref{eq:comp_left_right_compatible} is equivalent to \(V=\tilde{V}\).
\end{proof}

\begin{lemma}
  \label{lemm:continuity_of_morphisms}
  Right or left quantum group homomorphisms are injective and continuous coactions.
\end{lemma}

\begin{proof}
  Equations \eqref{eq:Delta_R_via_V} and~\eqref{eq:Delta_L_via_V} show that left and right quantum group homomorphisms are injective.  We only prove continuity for right quantum group homomorphisms, the left case is analogous.  Let \(\Delta_R\colon C\to C\otimes A\) be a right quantum group homomorphism with associated bicharacter \(V\in\U\Mult(\hat{C}\otimes A)\).  We must show that the linear span of \(\Delta_R(C) (1\otimes A)\) is dense in \(C\otimes A\).  We may replace~\(C\) by the dense subspace of slices \((\hat{c}\mu\otimes\Id_C)\multunit^C\) for \(\mu\in\hat{C}\) and \(\hat{c}\in\hat{C}\), where \(\hat{c}\mu\in\hat{C}'\) is defined by \(\hat{c}\mu (x) \defeq \mu(x\hat{c})\) for \(\hat{c}\in\hat{C}\), \(\mu\in \hat{C}'\), and \(x\in \hat{C}\).  We have
  \[
  \bigl((\hat{c}\mu\otimes\Id_C\otimes\Id_A)(\Id_{\hat{C}}\otimes\Delta_{R})\multunit^C\bigr)(1\otimes a)
  = (\mu\otimes\Id_C\otimes\Id_A)\bigl(\multunit^C_{12} V_{13}(\hat{c}\otimes 1\otimes a)\bigr).
  \]
  Here \(V_{13}(\hat{c}\otimes 1\otimes a)\) ranges over a linearly dense subset of \(\hat{C}\otimes 1\otimes A\).  Hence we do not change the closed linear span if we replace this expression by \(\hat{c}\otimes 1\otimes a\).  This leads to
  \[
  (\mu\otimes\Id_C\otimes\Id_A)(\multunit^C_{12} \cdot (\hat{c}\otimes 1\otimes a)) =
  \bigl((\hat{c}\mu\otimes\Id_C)\multunit^C\bigr) \otimes a,
  \]
  and these elements span a dense subspace of \(C\otimes A\) as asserted.
\end{proof}

\section{Functors between coaction categories}
\label{sec:functors}

Let~\(\Cstcat\) denote the category of \(\Cst\)\nb-algebras with morphisms (non-degenerate \Star{}homomorphisms \(A\to\Mult(B)\)) as arrows.  For a locally compact quantum group \((A,\Comult_A)\), let \(\Cstcat(A)\) or \(\Cstcat(A,\Comult_A)\) denote the category of \(\Cst\)\nb-algebras with a continuous, injective \(A\)\nb-coaction, together with \(A\)\nb-equivariant morphisms as arrows.  Lemma~\ref{lemm:continuity_of_morphisms} shows that left and right quantum group homomorphisms provide objects of our category.

Let \(\Forget\colon \Cstcat(C)\to\Cstcat\) be the functor that forgets the \(C\)\nb-coaction.  We now describe quantum group homomorphisms using functors \(F\colon \Cstcat(C)\to\Cstcat(A)\) with \(\Forget\circ F=\Forget\).  In particular, we show that a right quantum group homomorphism induces such a functor.  The results in this section answer a question posed to us by Debashish Goswami.

\begin{theorem}
  \label{the:homomorphism_functor}
  Let \((C,\Comult_C)\) and \((A,\Comult_A)\) be locally compact quantum groups.  Functors \(F\colon \Cstcat(C)\to\Cstcat(A)\) with \(\Forget\circ F=\Forget\) are in natural bijection with right quantum group homomorphisms from~\(C\) to~\(A\).

  More precisely, let \(\gamma\colon D\to D\otimes C\) be a continuous coaction of \((C,\Comult_C)\) on a \(\Cst\)\nb-algebra~\(D\) and let \(\Delta_R\colon C\to C\otimes A\) be a right quantum group homomorphism.  Then there is a unique continuous coaction~\(\alpha\) of \((A,\Comult_A)\) on~\(D\) such that the following diagram commutes:
  \begin{equation}
    \label{eq:induced_coaction}
    \begin{gathered}
      \xymatrix@C+2em{
        D \ar[r]^{\gamma} \ar[d]_{\alpha}&
        D\otimes C \ar[d]^{\Id_D \otimes\Delta_R}\\
        D\otimes A \ar[r]_{\gamma\otimes\Id_A}&
        D\otimes C\otimes A.
      }
    \end{gathered}
  \end{equation}
  If a morphism \(D\to D'\) between two \(\Cst\)\nb-algebras with continuous \(C\)\nb-coactions is \(C\)\nb-equivariant, then it is \(A\)\nb-equivariant as well, so that this construction is a functor \(F\colon \Cstcat(C)\to\Cstcat(A)\) with \(\Forget\circ F=\Forget\).  Conversely, any such functor is of this form for some right quantum group homomorphism~\(\Delta_R\).
\end{theorem}

\begin{proof}
  A map~\(\alpha\) making~\eqref{eq:induced_coaction} commute is unique if it exists because \(\gamma\otimes\Id_A\) is injective.  Existence means \((\Id_D\otimes\Delta_R)\gamma(D)\subseteq (\gamma\otimes\Id_A)(D\otimes A)\).  Let \(\Delta_L\colon C\to A\otimes C\) be the left quantum group homomorphism satisfying~\eqref{eq:comp_left_right_homomorphism}.  We compute
  \begin{align*}
    (\Id_D\otimes\Delta_R\otimes\Id_C)(\gamma\otimes\Id_C)\gamma
    &= (\Id_D\otimes\Delta_R\otimes\Id_C)(\Id_D\otimes\Comult_C)\gamma
    \\&= (\Id_D\otimes\Id_C\otimes\Delta_L)(\Id_D\otimes\Comult_C)\gamma
    \\&= (\Id_D\otimes\Id_C\otimes\Delta_L)(\gamma\otimes\Id_C)\gamma
    \\&= (\gamma\otimes\Delta_L)\gamma =
    (\gamma\otimes\Id_{A\otimes C})(\Id_D\otimes\Delta_L)\gamma,
  \end{align*}
  where the first and third equality use that~\(\gamma\) is coassociative, the second one uses~\eqref{eq:comp_left_right_homomorphism}, and the fourth one is trivial.

  Thus \((\Id_D\otimes\Delta_R\otimes\Id_C)(\gamma\otimes\Id_C)\) maps \(\gamma(D)\) into \((\gamma\otimes\Id_{A\otimes C})(D\otimes A\otimes C)\).  Since it also maps \(1_D\otimes C\) into \((\gamma\otimes\Id_{A\otimes C})(D\otimes A\otimes C)\) and \(\gamma(D)\cdot (1_D\otimes C)\) is dense in \(D\otimes C\) by the continuity of~\(\gamma\), \((\Id_D\otimes\Delta_R\otimes\Id_C)(\gamma\otimes\Id_C)\) maps \(D\otimes C\) into \((\gamma\otimes\Id_{A\otimes C})(D\otimes A\otimes C)\).  Thus \((\Id_D\otimes\Delta_R)\gamma(D)\subseteq (\gamma\otimes\Id_A)(D\otimes A)\) as desired.

  The second diagram in~\eqref{eq:intro_right_homomorphism} and several applications of~\eqref{eq:induced_coaction} imply
  \[
  (\gamma\otimes\Id_{A\otimes A})\circ (\alpha\otimes\Id_A)\circ \alpha =
  (\gamma\otimes\Id_{A\otimes A})\circ (\Id_D\otimes\Comult_A)\circ \alpha.
  \]
  Since \(\gamma\otimes\Id_{A\otimes A}\) is injective, \(\alpha\) satisfies \((\alpha\otimes\Id_A)\circ \alpha = (\Id_D\otimes\Comult_A)\circ \alpha\).  The map~\(\alpha\) is injective as well.  We check that~\(\alpha\) is continuous.  Since~\(\Delta_L\) and~\(\gamma\) are continuous,
  \begin{multline*}
    (\Id_D\otimes\Delta_L)\gamma(D)\cdot (1_D\otimes A\otimes C)\\
    \begin{aligned}
      &= (\Id_D\otimes\Delta_L)\gamma(D)\cdot (1_D\otimes \Delta_L(C))\cdot (1_D\otimes A\otimes 1_C)
      \\&= (\Id_D\otimes\Delta_L)(\gamma(D)\cdot 1_D\otimes C)
      \cdot (1_D\otimes A\otimes 1_C)
      \\&= (\Id_D\otimes\Delta_L)(D\otimes C) \cdot (1_D\otimes A\otimes 1_C)
      \\&= (D\otimes 1_{A\otimes C})\cdot (1_D\otimes\Delta_L(C))\cdot (1_D\otimes A\otimes 1_C)
      = D\otimes A\otimes C.
    \end{aligned}
  \end{multline*}
  Hence
  \begin{multline*}
    (\gamma\otimes\Id_{A\otimes C})(\alpha(D)\otimes C\cdot (1_D\otimes A\otimes C))
    \\=
    (\gamma\otimes\Id_{A\otimes C})( (\Id_D\otimes\Delta_L)\gamma(D)\cdot (1\otimes A\otimes C))
    =
    (\gamma\otimes\Id_{A\otimes C})(D\otimes A\otimes C).
  \end{multline*}
  Since \(\gamma\otimes\Id_{A\otimes C}\) is injective, we get \((\alpha(D)\otimes C)\cdot (1_D\otimes A\otimes 1_C) = D\otimes A\otimes C\).  This implies \(\alpha(D)\cdot (1_D\otimes A)= D\otimes A\) and hence that~\(\alpha\) is continuous.

  It is easy to see that a \(C\)\nb-equivariant map \(D\to D'\) remains \(A\)\nb-equivariant for the induced \(A\)\nb-coactions.  Thus we get a functor \(F\colon \Cstcat(C)\to\Cstcat(A)\) with \(\Forget\circ F=\Forget\) from a right quantum group homomorphism.

  Now let, conversely, \(F\colon \Cstcat(C)\to\Cstcat(A)\) be a functor with \(\Forget\circ F=\Forget\), that is, \(F\) maps a continuous \(C\)\nb-coaction \(\gamma\colon D\to D\otimes C\) on some \(\Cst\)\nb-algebra in a natural way to a continuous \(A\)\nb-coaction \(F(\gamma)\colon D\to D\otimes A\) on the same \(\Cst\)\nb-algebra.  We claim that~\(F\) must come from some right quantum group homomorphism \(\Delta_R\colon C\to C\otimes A\) by the above construction.

  When we apply~\(F\) to the coaction \(\Comult_C\colon C\to C\otimes C\), we get an \(A\)\nb-coaction \(\Delta_R\colon C\to C\otimes A\).  Being a coaction, it makes the second diagram in~\eqref{eq:intro_right_homomorphism} commute.  We will see later that the first diagram in~\eqref{eq:intro_right_homomorphism} also commutes.  First we use naturality to show that~\eqref{eq:induced_coaction} with \(\alpha = F(\gamma)\) commutes for any coaction of~\(C\), so that~\(\Delta_R\) determines the functor~\(F\).

  To begin with, we consider the coaction \(\Comult_C\oplus\Comult_C\colon C\oplus C\to (C\oplus C)\otimes C\).  Since the coordinate projections \(\pi_1,\pi_2\colon C\oplus C\to C\) are \(C\)\nb-equivariant, they are \(A\)\nb-equivariant with respect to \(F(\Comult_C\oplus\Comult_C)\) and \(F(\Comult_C)=\Delta_R\).  This already implies that \(F(\Comult_C\oplus\Comult_C)=\Delta_R\oplus\Delta_R\).

  Next we consider the coaction \(\Id_{\Comp(\Hils)} \otimes \Comult_C\) on \(\Comp(\Hils)\otimes C\).  For any projection \(P\in\Comp(\Hils)\), we get a \(C\)\nb-equivariant morphism \(C\oplus C\to \Comp(\Hils)\otimes C\), \((a,b)\mapsto aP + b(1-P)\).  Since we already know the \(A\)\nb-coaction \(F(\Delta_C\oplus \Delta_C)\), the induced \(A\)\nb-coaction on \(\Comp(\Hils)\otimes C\) maps \(P\otimes a\mapsto P\otimes \Delta_R(a)\).  Since this holds for all projections~\(P\) and since these projections generate \(\Comp(\Hils)\), we get \(F(\Id_{\Comp(\Hils)} \otimes \Comult_C) = \Id_{\Comp(\Hils)} \otimes \Delta_R\).

  Now consider a general coaction \(\gamma\colon D\to D\otimes C\).  Then~\(\gamma\) is \(C\)\nb-equivariant with respect to the coaction \(\Id_D\otimes\Comult_C\) on \(D\otimes C\).  Let \(\pi\colon D\to\Comp(\Hils)\) be a morphism coming from a faithful representation of~\(D\) on some Hilbert space~\(\Hils\).  The injective, \(C\)\nb-equivariant morphism \((\pi\otimes\Id_C) \circ\gamma\colon D\to \Comp(\Hils)\otimes C\) remains \(A\)\nb-equivariant with respect to the coactions \(F(\gamma)\) on~\(C\) and \(F(\Id_{\Comp(\Hils)}\otimes\Comult_C) = \Id_{\Comp(\Hils)}\otimes\Delta_R\) on \(\Comp(\Hils)\otimes C\).  This means that~\eqref{eq:induced_coaction} commutes with \(\alpha\defeq F(\gamma)\).  Finally, specialising~\eqref{eq:induced_coaction} to the coaction~\(\Delta_C\) on~\(C\) shows that the first diagram in~\eqref{eq:intro_right_homomorphism} commutes.  Thus~\(\Delta_R\) is a right quantum group homomorphism that generates~\(F\).  The construction also shows that~\(\Delta_R\) is unique.
\end{proof}

\begin{proposition}
  \label{pro:compose_right}
  Let \(\alpha\colon C\to C\otimes A\) and \(\beta\colon A\to A\otimes B\) be two right quantum group homomorphisms and let \(F_\alpha\colon \Cstcat(C)\to\Cstcat(A)\) and \(F_\beta\colon \Cstcat(A)\to\Cstcat(B)\) be the associated functors.  Then \(F_\beta\circ F_\alpha = F_\gamma\), where \(\gamma\colon C\to C\otimes B\) is the unique right quantum group homomorphism \(\gamma\colon C\to C\otimes B\) that makes the following diagram commute:
  \begin{equation}
    \label{eq:diag_for_comp_morph}
    \begin{gathered}
      \xymatrix@C+2em{
        C \ar[r]^{\alpha} \ar[d]_{\gamma}&
        C\otimes A \ar[d]^{\Id_C \otimes\beta}\\
        C\otimes B \ar[r]_{\alpha\otimes\Id_B}&
        C\otimes A\otimes B.
      }
    \end{gathered}
  \end{equation}
  Furthermore, the bicharacter associated to~\(\gamma\) is the composition of the bicharacters associated to \(\beta\) and~\(\alpha\).
\end{proposition}

\begin{proof}
  Theorem~\ref{the:homomorphism_functor} shows that \(F_\alpha\) maps the coaction~\(\Delta_C\) to~\(\alpha\).  This is mapped by~\(F_\beta\) to the unique morphism making~\eqref{eq:diag_for_comp_morph} commute.  Thus \(F_\beta\circ F_\alpha\) maps~\(\Delta_C\) to~\(\gamma\), forcing \(F_\beta\circ F_\alpha = F_\gamma\).

Theorem~\ref{the:homomorphism_functor} yields a unique continuous right coaction~\(\gamma\) of \((B,\Comult_B)\) on \((C,\Comult_C)\) making~\eqref{eq:diag_for_comp_morph} commute.  It is not hard to show that this is a right quantum group homomorphism.  Anyway, we want to convince ourselves that this construction corresponds to the composition of bicharacters.

  Since slices of~\(\multunit\) by continuous linear functionals on~\(\hat{C}\) generate a dense subspace of~\(C\), the diagram~\eqref{eq:diag_for_comp_morph} commutes if and only if
  \[
  (\Id_{\hat{C}}\otimes\Id_C\otimes\beta)(\Id_{\hat{C}}\otimes\alpha)(\multunit^C)
  = (\Id_{\hat{C}}\otimes\alpha\otimes\Id_B)(\Id_{\hat{C}}\otimes\gamma)(\multunit^C).
  \]
  Equation~\eqref{eq:def_V_via_right_homomorphism} implies \(\Id_{\hat{C}}\otimes\alpha(\multunit^C) = \multunit^C_{12} \Linkunit{C}{A}_{13}\), and \(\Id_{\hat{C}}\otimes\Id_A\otimes\beta\) maps this to the element represented by the unitary operator
  \[
  \Multunit^C_{12} \mathbb{V}^{A\to B}_{34} \mathbb{V}^{C\to A}_{13} (\mathbb{V}^{A\to B}_{34})^* = \Multunit^C_{12} \mathbb{V}^{C\to A}_{13} \mathbb{V}^{C\to B}_{14}
  \]
  by \eqref{eq:Delta_R_via_V} and~\eqref{eq:link_unit_comp_exist}.  Thus
  \[
  (\Id_{\hat{C}}\otimes\Id_C\otimes\beta)(\Id_{\hat{C}}\otimes\alpha)(\multunit^C)
  = \multunit^C_{12} \Linkunit{C}{A}_{13} \Linkunit{C}{B}_{14},
  \]
  where \(\Linkunit{C}{B} \defeq \Linkunit{A}{B} * \Linkunit{C}{A}\).  Let~\(\tilde{V}\) be the bicharacter associated to~\(\gamma\).  Equation~\eqref{eq:def_V_via_right_homomorphism} implies
  \[
  (\Id_{\hat{C}}\otimes\alpha\otimes\Id_B)(\Id_{\hat{C}}\otimes\gamma)(\multunit^C)
  = (\Id_{\hat{C}}\otimes\alpha\otimes\Id_B)(\multunit^C_{12} \tilde{V}_{13})
  = \multunit^C_{12} \Linkunit{C}{A}_{13} \tilde{V}_{14}.
  \]
  Hence~\eqref{eq:diag_for_comp_morph} commutes if and only if \(\tilde{V} = \Linkunit{C}{B}\).
\end{proof}

\begin{example}
  \label{ex:compositon_of_link_unit_induced_by_strong_morph}
  Let \(\Linkunit{C}{A}\in\U\Mult(\hat{C}\otimes A)\) and \(\Linkunit{A}{B}\in\U\Mult(\hat{A}\otimes B)\) be bicharacters.

  Assume first that~\(\Linkunit{A}{B}\) comes from a Hopf \(^*\)\nb-homomorphism \(f\colon A\to B\), that is, \(\Linkunit{A}{B} = (\Id\otimes f)(\multunit^A)\).  Let~\(\alpha\) be the right quantum group homomorphism from~\(C\) to~\(A\) associated to~\(\Linkunit{C}{A}\).  The right quantum group homomorphism from~\(A\) to~\(B\) associated to~\(\Linkunit{A}{B}\) is \(\beta\defeq (\Id_A\otimes f)\Comult_A\).  The following computation shows that \(\gamma=(\Id_C\otimes f)\alpha\) satisfies~\eqref{eq:diag_for_comp_morph}:
  \begin{align*}
    (\Id_{\hat{C}}\otimes\Id_C\otimes\beta)(\Id_{\hat{C}}\otimes\alpha)\multunit^C
    &= (\Id_{\hat{C}}\otimes\Id_C\otimes\Id_A\otimes f)
    (\Id_{\hat{C}}\otimes\Id_C\otimes\Comult_A)\multunit^C_{12}\Linkunit{C}{A}_{13}\\
    &= (\Id_{\hat{C}}\otimes\Id_C\otimes\Id_A\otimes f)\multunit^C_{12}
    \Linkunit{C}{A}_{13}\Linkunit{C}{A}_{14}\\
    &= (\Id_{\hat{C}}\otimes\Id_C\otimes\Id_A\otimes f)(\Id_{\hat{C}}\otimes\alpha\otimes\Id_B)
    \multunit^C_{12}\Linkunit{C}{A}_{13}\\
    &= (\Id_{\hat{C}}\otimes\alpha\otimes\Id_B)(\Id_{\hat{C}}\otimes (\Id_C\otimes f)\alpha)\multunit^C;
  \end{align*}
  the first step uses~\eqref{eq:def_V_via_right_homomorphism}; the second step uses~\eqref{eq:linking_unitary_comult_A}; the third and the last step use~\eqref{eq:def_V_via_right_homomorphism}.  Proposition~\ref{pro:compose_right} yields \(\beta*\alpha=(\Id_C\otimes f)\alpha\).  Hence the composition \(\Linkunit{A}{B} * \Linkunit{C}{A}\) is \((\Id_C\otimes f)\Linkunit{C}{A}\).
\end{example}

\begin{example}
  \label{ex:compositon_of_link_unit_induced_by_strong_morph_II}
  Now assume that~\(\Linkunit{C}{A}\) is constructed from a Hopf \(^*\)\nb-homomorphism \(f\colon \hat{A}\to \hat{C}\), that is, \(\Linkunit{C}{A} = (f\otimes\Id_A)(\multunit^A)\).  Then the composition \(\Linkunit{C}{B}\) is \((f\otimes\Id)(\Linkunit{A}{B})\).  This follows easily from Example~\ref{ex:compositon_of_link_unit_induced_by_strong_morph} because \(C\mapsto \hat{C}\) is a contravariant functor on bicharacters.
\end{example}

\begin{proposition}
  \label{pro:functorial_corepresentations}
  A right quantum group homomorphism from~\(C\) to~\(A\) induces a natural map from Hilbert space corepresentations of~\(C\) to corepresentations of~\(A\) on the same Hilbert space.
\end{proposition}

\begin{proof}
  Instead of giving a direct proof, we reduce this to Theorem~\ref{the:homomorphism_functor}.  A corepresentation of~\(C\) on~\(\Hils\) is equivalent to a coaction of~\(C\) on the \(\Cst\)\nb-algebra \(\Comp(\Hils\oplus\C)\) that restricts to the trivial coaction on the corner \(\C=\Comp(\C)\) and hence leaves the corner \(\Comp(\Hils)\) invariant (see \cite{Baaj-Skandalis:Hopf_KK}*{Proposition 2.8}).  A right quantum group homomorphism allows us to turn this \(C\)\nb-coaction on \(\Comp(\Hils\oplus\C)\) into an \(A\)\nb-coaction on \(\Comp(\Hils\oplus\C)\), which still fixes the corner~\(\C\) by functoriality and hence comes from an \(A\)\nb-corepresentation on~\(\Hils\).
\end{proof}

\end{document}